\documentclass[10pt]{article}

\usepackage[english]{babel}
\usepackage{graphicx,epstopdf,epsfig}
\usepackage{amsfonts,epsfig,fancyhdr,graphics, hyperref,amsmath,amssymb, amsthm}

\usepackage{bbm}
\usepackage{tikz}
\usepackage{xcolor}
\usepackage{enumerate}

\providecommand{\keywords}[1]{{\textit{Keywords and phrases:}} #1}
\providecommand{\subjclass}[1]{{\textit{AMS Subject [2020]:}} #1}
\date{\today}

\newtheorem{theorem}{Theorem}[section] 
\newtheorem{lemma}[theorem]{Lemma}
\newtheorem{prop}[theorem]{Proposition}

\newtheorem{corollary}[theorem]{Corollary}
\newtheorem{remark}[theorem]{Remark}
\newtheorem{example}[theorem]{Example}

\begin{document}
\bibliographystyle{plain}
\setcounter{page}{1}
\thispagestyle{empty}

\title{Kemeny's constant and Braess cliques in graphs}

\author{Jane Breen\thanks{Faculty of Science, Ontario Tech University, Oshawa, ON L1K1J9, Canada ({jane.breen@ontariotechu.ca}).
J. Breen is supported by the Natural Sciences and Engineering Research Council of Canada (NSERC) Discovery Grant RGPIN--2021--03775.} \and Emma deBlieck\thanks{Department of Mathematics, Redeemer University, Ancaster, ON L9A 3V8, Canada ({emma.deblieck@gmail.com, kvanderm@redeemer.ca}). Supported in part by an NSERC USRA (EdB) and NSERC Discovery Grant RGPIN--2022--05137 (KVM). 
} \and Kevin N. Vander Meulen\footnotemark[2]
 }

\maketitle

\begin{abstract}
    Kemeny's constant is used as a measure of the average travel time on a graph. Braess' paradox for graphs is the observation that in some graphs, when an edge is added, Kemeny's constant increases. We introduce the notion of a Braess clique $K_\ell$, a clique that when inserted into a graph on an independent set of $\ell$ vertices, will create an increase in Kemeny's constant. In this context, a Braess edge is a Braess $K_2$. We provide examples of graphs that have a Braess $K_\ell$ for $\ell\geq 3$. We observe that almost every connected planar labelled graph has a Braess $K_\ell$ for each $\ell\geq 3$. We also explore the relationship between Braess edges and Braess cliques in graphs. 
\end{abstract}

\medskip

\noindent

\newcommand{\K}{\mathcal{K}}

\section{Introduction}

Kemeny's constant is an increasingly-popular spectral invariant in graph theory which is used as a measure of the overall connectivity of a graph. 
Originating from Markov chain theory, Kemeny's constant captures the behaviour of a random walk on the vertices of a graph, and gives a numerical measure $\mathcal{K}(G)$ of the expected length of a random trip between two randomly-chosen vertices of the graph $G$, where the starting and ending vertices are chosen with probability proportional to their degrees. Kemeny's constant is well known to behave somewhat unexpectedly or against intuition at times. An example of this is the exhibition of \emph{Braess' paradox}. This paradox originated with regard to road network traffic in the 1920s, when it was observed that the addition of a new road to an existing road network can have an overall negative effect on the flow of traffic through the network; i.e., slowing it down. Analogously, there are instances where adding an edge to a graph $G$ can produce a new graph $\widehat{G}$ for which $\mathcal{K}(\widehat{G}) \geq \mathcal{K}(G)$, meaning that the expected length of a random trip in the new graph is longer with the additional edge than previously, or the new graph is less `well-connected' than the old one. Such a non-edge in a graph is called a \emph{Braess edge} of the graph. 

The study of Braess edges began in \cite{kirkland2016kemeny}, where it was shown that almost every tree has a Braess edge, by showing that if a tree has two pendent twin vertices, adding an edge between them increases Kemeny's constant (see Fig.~\ref{fig:small_example}). In~\cite{ciardo2020braess}, it was shown that every graph with pendent twins has a Braess edge between those twin vertices. Other families of graphs with Braess edges have been studied; in \cite{hu2019complete}, complete multipartite graphs are studied, and conditions under which a non-edge between vertices in the same partite set is a Braess edge are given, as well as a characterization of such graphs where \emph{every} non-edge is a Braess edge (these graphs are called \emph{Braess graphs}). In \cite{kim2022families}, graphs with long pendent paths are considered, and it is considered whether the addition of the edge between the endpoints of these paths increases Kemeny's constant. In \cite{jang2025kemeny}, the authors count the number of Braess edges in certain families of trees, such as paths, brooms, and spider graphs. In \cite{faught20221}, the authors consider the presence of \emph{Braess sets} in a graph, which are sets of non-edges in a graph $G$ such that the addition of all edges in these sets cause an increase in Kemeny's constant. 

In this article, we focus on Braess sets that form a clique. 
We refer to such an instance as a \emph{Braess clique}: that is, an independent set of vertices in a graph $G$ such that adding all possible edges between these vertices causes an increase in Kemeny's constant.  In general, since a complete graph has the lowest Kemeny's constant of all graphs
on a fixed number of vertices, one would expect that inserting a clique into a graph
would significantly reduce Kemeny's constant. That there exists Braess cliques
is counter-intuitive in this context, but non-trivial Braess cliques were already observed in \cite{faught20221} in the context of pendent vertices.

The goal is to further our understanding of the structures in a graph that cause Kemeny's constant to increase. %be large. 
While the existence of Braess edges is initially counter-intuitive, it can be argued that adding an edge between pendent twin vertices creates a substructure in the graph in which a random walker spends more time,  
thus increasing the overall expected length of a random trip between vertices in a graph. If we add more edges, does this amplify the same behaviour, or does the increased number of edges mitigate it?

\begin{figure}
\begin{center}
\begin{tikzpicture}[vtx/.style={circle,fill,inner sep=1.6pt},
    bubble/.style={draw,very thick,circle,minimum size=1.5cm},
    ed/.style={line width=1pt}]
    
\begin{scope}
  \node[vtx] (s1) at (0,0) {};
  \node[vtx] (s2) at (0.707, 0.707) {};
  \node[vtx] (s3) at (0.707, -0.707) {};
  \node[vtx] (s4) at (-1, 0) {};
  \draw[ed] (s4)--(s1)--(s2); 
  \draw[ed] (s1)--(s3);
  
  \node (label1) at (-0.4, -0.9) {$G$};
\end{scope}

\begin{scope}[xshift=5cm]
  \node[vtx] (c1) at (0,0) {};
  \node[vtx] (c2) at (0.707, 0.707) {};
  \node[vtx] (c3) at (0.707, -0.707) {};
  \node[vtx] (c4) at (-1, 0) {};
  \draw[ed] (c4)--(c1)--(c2); 
  \draw[ed] (c1)--(c3);
  \draw[very thick, orange, dashed] (c3)--(c2);
  
  \node (label2) at (-0.4, -0.9) {$\widehat{G}$};
\end{scope}

\end{tikzpicture}
\end{center}
\label{fig:small_example}
\caption{An example of a graph with a Braess edge: $\mathcal{K}(G) = \frac{5}{2}$ and $\mathcal{K}(\widehat{G})=\frac{61}{24}.$}
\end{figure}
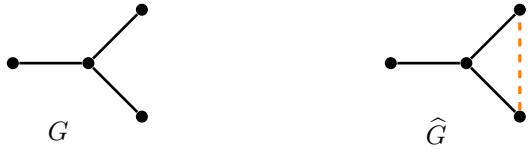

The remainder of this article is arranged as follows: In Section 2, we give necessary mathematical preliminaries with background on Markov chains, the accessibility index, and Kemeny's constant. In Section 3, we discuss equitable partitions in stochastic matrices, and determine an expression for the change in Kemeny's constant when adding a clique of order $k$ to an independent set of $k$ twin vertices. In Section 4, we apply these results to two families of graphs: graphs with $k$ pendent vertices attached to some specified vertex $v$, and complete bipartite graphs. In Section 5, we discuss the interplay of Braess edges and Braess cliques, investigating whether adding $k$ edges which are individually Braess will form a Braess clique, and whether every Braess clique must contain at least one Braess edge. We provide some constructions of graph families in this section to illustrate the observations.

\section{Mathematical preliminaries}

\subsection{Markov chains}
A finite, discrete-time, time-homogeneous Markov chain is a model of a stochastic dynamical system which transitions between states from a finite state space $\{s_1, s_2, \ldots, s_n\}$ in discrete time-steps, with the additional property that the state of the system in the next time-step depends 
on the current state of the system (this is known as the \emph{Markov property}). The behaviour of such a system is described  
using the probability transition matrix, an $n\times n$ row-stochastic matrix denoted $T=[t_{ij}]$ for which $t_{ij}$ 
is the probability of transitioning to $s_j$ in a single time-step, given that the system is currently in state $s_i$. Using the Markov property and properties of matrix multiplication, 
one can see that the $(i,j)$ entry of $T^k$ is the probability of the system being in state $s_j$ in exactly $k$ steps, given that it is currently in state $s_i$. 

If the transition matrix $T$ is primitive (i.e.~there exists some positive integer $m$ such that $T^m$ is a positive matrix) then due to the Perron-Frobenius theorem, $\lim_{k\to\infty} T^k$ exists and is equal to the rank one matrix $\mathbbm{1}w^\top$, where $\mathbbm{1}$ is the all-ones vector (a right eigenvector of $T$ corresponding to the Perron value $\rho(T)=1$), and $w^\top$ is the corresponding left Perron vector, normalized so that it sums to 1 and can be interpreted as a probability distribution. Indeed, since each row of $T^k$ converges to this distribution $w^\top$, this vector is interpreted as the long-term distribution of the Markov chain, and is named the \emph{stationary distribution vector} of the chain. That is, $w_i$ is interpreted as the long-term probability of the chain occupying state $s_i$. Note that in the case that $T$ is irreducible but not primitive, there is some periodic behaviour evident in the system, and while convergence to $\mathbbm{1}w^\top$ does not occur, the stationary distribution vector still exists and is unique, and instead may be interpreted entry-wise as the long-term proportion of time spent in each state. For more details on this distinction, see \cite{seneta}; note that in this article, when considering random walks on graphs, the only graphs on which a random walk exhibits periodic behaviour are bipartite graphs.

The short-term behaviour of a Markov chain is quantified using the \emph{mean first passage times}, denoted $m_{i,j}$ for $i,j\in\{ 1,\ldots, n\}$, representing the expected time it takes to reach state $s_j$ for the first time, given that the chain starts in state $s_i$. The matrix of mean first passage times, $M = [m_{i,j}]$, is the unique solution to the recursive-type expression
\[M = T(M-M_{dg}) + J,\]
where $A_{dg}$ denotes a diagonal matrix obtained from a matrix $A$ by zeroing out all off-diagonal entries, and $J=J_n$ denotes a square all-ones matrix of order $n$. Note that we take the convention that $m_{i,i}$ denotes the \emph{mean first return time} to state $s_i$, in which case it is known that 
$m_{i,i} = \frac{1}{w_i}$. 

While several formulas exist for the computation of mean first passage times, both individually and for the mean first passage matrix as a whole, in this article we will make use of the \emph{group inverse} of the singular matrix $I-T$. For any complex $n\times n$ singular matrix $A$ for which $0$ is a semisimple eigenvalue (i.e. the algebraic and geometric multiplicities coincide), the group inverse of $A$ is the unique matrix $X$ which satisfies the following:
\[AXA = A; \quad XAX = X; \quad AX=XA.\]
The group inverse of $A$ is denoted $A^\#$. In this article we consider $(I-T)^\#$, from which, as stated by Meyer in \cite{meyer1975role}, ``the answer to every important question about the [Markov] chain can be obtained''. For example, the mean first passage matrix may be obtained as 
\[M = (I-(I-T)^\# + J[(I-T)^\#]_{dg})W^{-1},\]
where $W$ is the diagonal matrix whose $i^{th}$ diagonal entry is $w_i$.

Let $T$ be an irreducible transition matrix, with stationary vector $w$ and mean first passage matrix $M=[m_{i,j}]$. Fix an index $j$, and consider the quantity
\[\alpha_T(j) = \sum_{\substack{i=1\\i\neq j}}^n w_im_{i,j}.\]
This is known as the \emph{accessibility index} for state $s_j$, as it can be interpreted as the expected length of a random trip in the Markov chain, starting from stationarity, that terminates at state $s_j$. This was first considered in \cite{kirkland2016randomwalk}, and it was noted that it is related to the concept of \emph{random walk centrality} (see \cite{noh2004random}).

Now fix an index $i$, and consider the analogous expression 
\[\kappa_T(i) = \sum_{\substack{j=1\\j\neq i}}^n w_jm_{i,j}.\]
This admits an interpretation in terms of the expected length of a random trip in the Markov chain with fixed starting state $s_i$, and destination state chosen randomly with respect to the stationary distribution. Remarkably, this was shown in \cite{kemenysnell} to be independent of $i$; that is, the value of $\kappa_T(i)$ is constant, $i=1, \ldots, n$. As such, it is named \emph{Kemeny's constant}, and denoted by $\mathcal{K}(T)$. Note that since $\sum_i w_i = 1$, one can manipulate the above expression and write 
\[\mathcal{K}(T) = \sum_{\substack{i, j=1 \\j\neq i}}^n w_i m_{i,j} w_j,\]
which allows the interpretation of Kemeny's constant in terms of the expected length of a random trip in the chain, where both starting state and destination state are chosen randomly (with respect to the stationary distribution). As such, Kemeny's constant can be seen as a measure of overall connectivity of a graph. It can be observed that Kemeny's constant is any entry of the vector $Mw - \mathbbm{1}$, and the vector of accessibility indices $\alpha^\top$ is given by $w^\top M - \mathbbm{1}^\top$. In \cite{kirkland2016randomwalk}, it was observed that this allows for the accessibility index vector to be seen as a ``partition'' of Kemeny's constant, since 
\begin{equation}\label{eq:alphaw}
\alpha^\top w = w^\top Mw - 1 = \mathcal{K}(T).
\end{equation}

The following theorem gives a spectral expression for Kemeny's constant.
\begin{theorem}[\cite{levene2002kemeny}]
Let $T$ be the transition matrix of an irreducible Markov chain with eigenvalues $1=\lambda_1,\lambda_2,\ldots, \lambda_n$, then
\begin{equation}\label{eq:Keig}
\mathcal{K}(T) = \sum_{j=2}^{n} \frac{1}{1-\lambda_j}.
\end{equation}
\end{theorem}

\subsection{Graph theory and random walks on graphs}
A \emph{graph} $G$ is a pair $(V(G), E(G))$ referred to as the vertex set and edge set of $G$, respectively. The \emph{order} of a graph is $|V(G)|$ and the size is $|E(G)|$. In an undirected graph, an edge is an unordered set $\{u, v\}$, where $u, v \in V(G)$. If $\{u,v\} \in E(G)$, we say that $u$ is \emph{adjacent} to $v$ and we also say that $u$ is a \emph{neighbour} of $v$, and vice versa. 
The \emph{degree} of a vertex $v$, denoted by $\deg(v)$ or $\deg_G(v)$, is the number of neighbours of $v$ in $G$. Given  $S\subseteq V(G)$, then $S$ forms a \emph{clique} if $\{u, v\}\in E(G)$ for all $u,v\in S$, and  $S$ is  an \emph{independent set}, or forms a \emph{coclique}, if for all $u, v\in S$, $\{u,v\}\notin E(G)$. A \emph{complete bipartite graph}, denoted $K_{a, b}$, is a graph consisting of two cocliques $A$ and $B$ of order $a$ and $b$, respectively, such that for all $u\in A$, $v\in B$, $\{u,v\} \in E(K_{a, b})$.  A \emph{star} graph $S_n$ on $n$ vertices is a complete bipartite graph $K_{1,n-1}$ with the vertex of degree $n-1$ called the \emph{centre} vertex. Other basic concepts of graph theory can be found in \cite{CL}.

Given a connected graph $G$ with vertices labelled $v_1, \ldots, v_n$, a random walk on the graph $G$ is a Markov chain for which the states are the vertices of $G$, and the transition probabilities are 
\[t_{i,j} = \left\{\begin{array}{cc} \frac{1}{\deg(v_i)}, & \mbox{if } \{v_i, v_j\} \in E(G);\\
0, & \mbox{otherwise.}\end{array}\right.\]
That is, the transition matrix is given by $T = D^{-1}A$, where $A$ is the $(0,1)$ adjacency matrix of $G$, and $D$ is the diagonal matrix of vertex degrees. This has the interpretation of modelling the movement of a random walker who, at any given time, occupies a vertex $v_i$ of $G$, and in the next time-step chooses to move to a neighbour of $v_i$ uniformly at random.  In this context, the stationary vector $w_G$ is well-known to be proportional to the degree vector of $G$ (in particular, $w_i=\frac{\deg(v_i)}{2|E(G)|}$), the accessibility index $\alpha_G(v)$ is a measure of how easily the vertex $v$ is reached in a random walk starting from stationarity, and Kemeny's constant is a measure of the length of a random trip in a graph. Note that we denote Kemeny's constant for the graph as $\mathcal{K}(G) = \mathcal{K}(D^{-1}A)$.

With the combination of the useful interpretation of $\mathcal{K}(G)$ in terms of the expected length of a random trip in the graph and its expression as a spectral parameter, Kemeny's constant has become a subject of great interest and investigation in the field of spectral graph theory. 

The connection of Kemeny's constant with \emph{resistance distances} was first made in general in \cite{wang2017kemeny}, where it was shown that for a graph with adjacency matrix $A$, degree vector $d$ with $d_i=\deg(v_i)$, and number of edges $m$,
\[\mathcal{K}(D^{-1}A) = \frac{d^\top R d}{4m},\] 
where $R=[r_{i,j}]$ is a matrix of resistance distances. For a given vertex $v$ of a connected graph $G$, the \emph{moment} of $v$ in $G$ is 
\begin{equation}\label{eq:moment}
\mu_G(v) = \sum_{u\in V(G)} \deg(u) r_{u, v}.
\end{equation}

While the resistance matrix $R$ may be written in terms of the pseudoinverse of the Laplacian matrix of the graph, a more structural result relates $r_{i,j}$ to a number of spanning trees and forests of $G$ of a certain type.

\begin{prop}[\cite{kirkland2016kemeny}]
Let $G$ be a connected, undirected graph on $n$ vertices with degree vector $d$. Let $m$ be the number of edges of $G$, and let $\tau$ be the number of spanning trees of $G$. For each $i,j \in\{ 1, \ldots, n\}$ with $i\neq j$, let $f_{i,j}$ denote the number of spanning 2-forests separating $v_i$ and $v_j$; that is, the number of spanning forests of $G$ consisting of two trees, one containing $v_i$ and the other containing $v_j$. Set $f_{i,i} = 0$. Letting $F = [f_{i,j}]$, 
\[\mathcal{K}(G) = \frac{d^\top F d}{4m\tau}.\]
\end{prop}

It was observed in \cite{chebotarev2020hitting} that $r_{i,j} = \frac{f_{i,j}}{\tau}$, and so $\mu_G(v)$ can also be written 
\begin{equation}\label{eq:momentTree}
\mu_G(v) = \sum_{u\in V(G)} \frac{\deg(u)}{\tau} f_{u, v}.
\end{equation} 
Both the resistance distance formula and the spanning 2-forest formula have been used extensively to investigate the value of Kemeny's constant, particularly in the context of highly-structured graphs such as trees (see \cite{ciardo2022kemeny, kirkland2016kemeny}) threshold graphs (see \cite{breen2025threshold}), and flower graphs (see \cite{faught2020resistance}), to name a few. These expressions, while complex, also lend themselves well to situations in which small changes are made to the graph, such as adding or removing an edge, and investigating the change in Kemeny's constant (see \cite{kirkland2016kemeny, ciardo2020braess}). Finally, their construction allows for a divide-and-conquer strategy when working with graphs which are easily decomposed, such as graphs with cut vertices or bridges (see \cite{faught20221, kim2022families, Breen2022Bridges}). We highlight one such result here. A \emph{1-separation} formula was developed in \cite{faught20221} for calculating Kemeny's constant if a graph has a certain type of cut-vertex. In particular, let $G_1 \oplus_v G_2$ be the graph obtained from two disjoint connected graphs $G_1$ and $G_2$ by identifying one vertex of each graph with the same label $v$.

\begin{theorem}[\cite{faught20221}]\label{thm:1sep}  If $G=G_1 \oplus_v G_2$ for some connected graphs $G_1$ and $G_2$ with $m_1$ and $m_2$ edges respectively, then
\begin{equation*} %\label{eq:1sepK}
\mathcal{K}(G)=\frac{m_1(\mathcal{K}(G_1)+\mu_{G_2}(v))+m_2(\mathcal{K}(G_2)+\mu_{G_1}(v))}{m_1+m_2}.\end{equation*}
\end{theorem}

It is shown in \cite{Breen2022Bridges} that for all $v\in V(G)$, 
\begin{equation}\label{eq:amu}\mu_G(v) = \alpha_G(v) + \mathcal{K}(G),\end{equation}
which allows the following restatement of the above theorem.

\begin{corollary}\label{cor:1sepA}
If $G=G_1 \oplus_v G_2$ for some connected graphs $G_1$ and $G_2$ with $m_1$ and $m_2$ edges respectively, then
\begin{equation*}%\label{eq:1sepA}
\mathcal{K}(G)= \mathcal{K}(G_1) + \mathcal{K}(G_2) + \frac{m_1\alpha_{G_2}(v)+m_2\alpha_{G_1}(v)}{m_1+m_2}.\end{equation*}
\end{corollary}

Further, the difference in Kemeny's constant when adding 
an edge to a graph with a 1-separation (described in \cite[Theorem 4.1]{jang2025kemeny}) can be formulated more simply in terms of accessibility index:

\begin{theorem}\cite{jang2025kemeny}
Let $G=G_1 \oplus_v G_2$ for some connected graphs $G_1$ and $G_2$ with $m_1$ and $m_2$ edges respectively, with $m=m_1+m_2$. If $\widetilde{G}$ (resp. $\widetilde{G}_1$) is the graph obtained from $G$ (resp. $G_1$) by inserting an edge in $G_1$, then 
\[K(\widetilde{G})-\K(G)=
\K(\widetilde{G}_1)-\K(G_1) + \frac{m_2}{m+1}\left[\alpha_{\widetilde{G}_1}(v) - \alpha_{G_1}(v)\right] + \frac{m_2}{m(m+1)} \left[\alpha_{G_1}(v) + \alpha_{G_2}(v)\right].\]
\end{theorem}

\begin{lemma}\label{lem:alpha}
    For any vertex $v$ in any connected graph $G$, 
    $$\alpha_G(v) \geq \frac{1}{2}.$$ Equality is obtained if and only if $G$ is a star graph $K_{1,n-1}$ and deg($v)=1$.
\end{lemma} 

\begin{proof} Let $G$ be a connected graph with $m$ edges and $v\in V(G).$
Since $m_{ij} \geq 1$ for all $i \neq j$, 
$$\alpha_G(v) \geq \frac{1}{2m} \sum_{i \neq v} d_i = \frac{1}{2m} \left( \sum_{i=1}^n d_i - d_v \right)$$
Since $d_v \leq n-1$, 
$$\alpha_G(v) \geq \frac{1}{2m} \left( 2m - (n-1) \right) = 1 - \frac{n-1}{2m} $$
Now, since $m \geq n-1$, 
$$\frac{n-1}{2m} \leq \frac{1}{2}$$
Therefore,
$$\alpha_G(v) \geq 1 - \frac{1}{2} = \frac{1}{2}.$$
Note that equality holds throughout if and only if $m=n-1$ and $d_v=n-1$ (and $m_{iv}=1$ for all $i\neq v$). 
\end{proof}

\section{Twin vertices, equitable partitions, and the accessibility index}

In the main result of this section, Theorem~\ref{thm:braess_clique_diff_acc}, we give an expression for the change in Kemeny's constant when adding a clique of order $k$ to an independent set of $k$ twin vertices in a graph. We also discuss more general results on how $\mathcal{K}(G)$ may be computed in graphs with twin vertices. To do this, we use the concept of \emph{equitable partitions} in a matrix. 

Recall that a set of vertices $\{v_1, v_2, \ldots, v_k\}$ is a set of twin vertices if pairwise they share the same neighbour set. Furthermore, 
 in a matrix representing a graph with a set of twin vertices, one can use  equitable partitions in order to produce a quotient matrix of smaller order whose eigenvalues are among those of the original matrix. In the next lemma, we describe the general result regarding equitable partitions; this may be more familiar to the reader in the context of the adjacency matrix of a graph (see \cite[Section 2.3]{brouwer2012spectra}).
\begin{lemma}\label{lem:equitable_partitions}
Let $A$ be a square block-partitioned matrix 
\[A = \left[\begin{array}{c|c|c|c}
A_{11} & A_{12} & \cdots & A_{1m}\\\hline
A_{21} & A_{22} & \cdots & A_{2m} \\\hline
\vdots & \vdots & \ddots & \vdots \\ \hline
A_{m1} & A_{m2} & \cdots & A_{mm}\end{array}\right],\]
with the property that each block $A_{ij}$ has constant row sums $q_{ij}$. That is, $A_{ij}\mathbbm{1}_{k_j} = q_{ij}\mathbbm{1}_{k_i}$, where $k_i$ is the order of the $i^{th}$ diagonal block of $A$. Then if $\lambda\in \mathbb{C}$ is an eigenvalue of the quotient matrix $Q = [q_{ij}]$, $\lambda$ is also an eigenvalue of $A$. Furthermore, if $\mathbf{x} \in \mathbb{C}^m$ is an eigenvector of $Q$ corresponding to $\lambda$, the vector 
\[\hat{\mathbf{x}}^\top = \left[\begin{array}{c|c|c|c} x_1 \mathbbm{1}_{k_1}^\top & x_2 \mathbbm{1}_{k_2}^\top & \cdots & x_m \mathbbm{1}_{k_m}^\top\end{array}\right]\] 
is an eigenvector of $A$ corresponding to $\lambda$.
\end{lemma}

Let $G$ be a graph of order $n$, and suppose that the vertices are ordered such that the first $k$ vertices are an independent set $U$ of twin vertices, that the common neighbourhood of $v_1, v_2, \ldots, v_k$ appear next in the ordering, followed by the remaining vertices in $G$. Suppose that the degree of each vertex in $U=\{v_1, v_2, \ldots, v_k\}$ is $r$. Then the 
transition matrix for the random walk on $G$ is 
\begin{equation}\label{eq:T_ind_set}
\arraycolsep=2.4pt\def\arraystretch{1.2}
T = \left[\begin{array}{c|c|c} O_{k, k} & \frac{1}{r}J_{k,r} & O_{k, n-r-k} \\[2pt]\hline
\mathbf{x}\mathbbm{1}_k^\top & A & B \\\hline
O_{n-r-k, k} & C & D \end{array}\right]
\end{equation}
where $\mathbf{x}$ is some vector, and $A, B, C, D$  matrices representing transitions between vertices in $\{v_{k+1}, \ldots, v_n\}$. 

Consider the equitable partition of this $n\times n$ matrix given by $\{\{v_1, \ldots v_k\}, \{v_{k+1}\}, \ldots \{v_{n}\}\}$ (note that the partition lines in the matrix above are not intended to represent an equitable partition, but rather the sets of vertices according to the order prescribed above). The quotient matrix is the $(n-k+1)\times(n-k+1)$ matrix 
\begin{equation}\label{eq:Q_ind_set}
Q = \left[\begin{array}{c|c|c} 0 & \frac{1}{r}\mathbbm{1}_r^\top & 0_{n-r-k}^\top \\[2pt]\hline
k\mathbf{x} & A & B \\\hline
0_{n-r-k} & C & D \end{array}\right].
\end{equation}
In the graph setting, we refer to this operation as the \emph{coalescence} of the set $U$.

In our next result, we consider the relationship between the values of Kemeny's constant for such a transition matrix and its quotient.

\begin{lemma}\label{lem:kemeny_quotient}
Let $G$ be a graph of order $n$ with an independent set of twin vertices of order $k$, with common degree $r$. Let $T$ be the transition matrix for the random walk on $G$, written in the form \eqref{eq:T_ind_set}, and let $Q$ be its quotient matrix. Let $\tilde{T}$ be the transition matrix obtained when adding all possible edges between the twin vertices, and $\tilde{Q}$ the quotient matrix of $\tilde{T}$. Then 
$\mathcal{K}(T) = \mathcal{K}(Q) + k-1 $
and
\[\mathcal{K}(\tilde{T}) = \mathcal{K}(\tilde{Q}) + \frac{(k-1)(r+k-1)}{r+k}.\]
\end{lemma}
\begin{proof}
Let $G$ be a graph of order $n$, and let $T$ be the transition matrix for the random walk on $G$, written in the form \eqref{eq:T_ind_set}, and $Q$ its quotient matrix as described in \eqref{eq:Q_ind_set}.

When adding a clique to the set of vertices $\{v_1, v_2, \ldots, v_k\}$, we achieve the transition matrix 
\begin{equation*}%\label{eq:T_clique}
\arraycolsep=2.4pt\def\arraystretch{1.4}
\tilde{T} = \left[\begin{array}{c|c|c} \frac{1}{r+k-1}(J_{k,k} - I_k) & \frac{1}{r+k-1}J_{k,r} & O_{k, n-r-k} \\[2pt]\hline
\mathbf{x}\mathbbm{1}_k^\top & A & B \\\hline
O_{n-r-k, k} & C & D \end{array}\right]
\end{equation*}
with corresponding quotient matrix
\begin{equation*}%\label{eq:Q_clique}
\tilde{Q} = \left[\begin{array}{c|c|c} \frac{k-1}{r+k-1}& \frac{1}{r+k-1}\mathbbm{1}_r^\top & 0_{n-r-k}^\top \\[2pt]\hline
k\mathbf{x} & A & B \\\hline
0_{n-r-k} & C & D \end{array}\right].
\end{equation*}

Let $e_1,e_2,\ldots, e_n$ be the columns of an order $n$ identity matrix. The $(k-1)$ linearly independent vectors
$(e_1-e_j)$ for $2\leq j\leq k$ 
are eigenvectors of $T$ corresponding to the eigenvalue $0$, and eigenvectors of $\tilde{T}$ corresponding to the eigenvalue $\frac{-1}{r+k-1}$. Furthermore, these eigenvectors must be orthogonal to any eigenvector of $T$ (or $\tilde{T}$, respectively) obtained from the quotient matrix $Q$ (respectively, $\tilde{Q}$), since by Lemma \ref{lem:equitable_partitions} those eigenvectors are constant on the first $k$ entries. Hence, the conclusions follow by 
\eqref{eq:Keig}.
\end{proof}

Next, we consider how the accessibility index of each state in the quotient Markov chain is related to the accessibility indices in the original Markov chain.

\begin{prop}\label{prop:acc_quotient}
Let $G$ be a graph of order $n$ and size $m$ with a set $U=\{u_1, u_2,\ldots, u_k\}$ of independent twin vertices with common degree $r$. Let $T$ be the transition matrix for the random walk on $G$, written in the form \eqref{eq:T_ind_set}, and $Q$ its quotient matrix as described in \eqref{eq:Q_ind_set}. Let $v_0$ be the first state of $Q$ corresponding to the coalescence of $U$. Then 
\begin{enumerate}[(a)]
\item $\alpha_T(u_1) = \alpha_T(u_2) = \cdots = \alpha_T(u_k)$.
\item For all $v\in V(G)\backslash U$, 
\[\alpha_Q(v) = \alpha_T(v).\]
\item For any $u\in U$, 
\[\alpha_Q(v_0) = \alpha_T(u) -\frac{2m(k-1)}{kr}.\]
\end{enumerate}
\end{prop}

\begin{proof}

\begin{enumerate}[(a)]
\item This follows from the definition of twin vertices.

\item The proof of this part is involved, and uses stochastic complements (see Appendix~\ref{app:stoch_comp} for necessary notation). Partition $T$ as in \eqref{eq:T_ind_set}, with 
\[
T = \left[ \begin{array}{c|c}
T_{11}&T_{12}\\ \hline 
T_{21}&T_{22} \end{array} 
\right]= 
\arraycolsep=2.4pt\def\arraystretch{1.2}
\left[\begin{array}{c|cc} O_{k, k} & \frac{1}{r}J_{k,r} & O_{k, n-r-k} \\[2pt]\hline
\mathbf{x}\mathbbm{1}_k^\top & A & B \\
O_{n-r-k, k} & C & D \end{array}\right].
\]
The quotient matrix $Q$ is partitioned similarly, with 
\[Q = \left[ \begin{array}{c|c}
Q_{11}&Q_{12}\\ \hline 
Q_{21}&Q_{22} \end{array} 
\right]= \left[\begin{array}{c|cc} 0 & \frac{1}{r}\mathbbm{1}_{r}^\top & 0_{n-r-k}^\top \\[2pt]\hline
k\mathbf{x} & A & B \\
0_{n-r-k} & C & D \end{array}\right].
\]
Note that the stationary vectors of $T$ and of $Q$, partitioned conformally, are 
\begin{eqnarray*}
w_T^\top & = & \frac{1}{2m}\left[\begin{array}{c|c} r\mathbbm{1}_k^\top & z^\top\end{array}\right] %\\
\qquad {\text{\ and \ }}\qquad 
w_Q^\top  =  \frac{1}{2m}\left[\begin{array}{c|c} kr & z^\top\end{array}\right],
\end{eqnarray*}
where $z$ is the vector of degrees of vertices in $V(G)\backslash U$.

We let $M^{(T)}$ and $M^{(Q)}$ denote the mean first passage matrices of $T$ and of $Q$, respectively, and consider certain entries of $w_T^\top M^{(T)}$ and $w_Q^\top M^{(Q)}$ to show that they are equal, in order to show that $\alpha_Q(v) = \alpha_T(v)$ for $v\in V(G)\backslash U$. In partitioned form,   
\begin{eqnarray*}
2m w_T^\top M_T & = & \left[\begin{array}{c|c} r\mathbbm{1}_k^\top & z^\top\end{array}\right] \left[\begin{array}{c|c} M_{11}^{(T)} & M_{12}^{(T)} \\ \hline M_{21}^{(T)} & M_{22}^{(T)}\end{array}\right]\\
& = & \left[\begin{array}{c|c} r\mathbbm{1}_k^\top M_{11}^{(T)} + z^\top M_{21}^{(T)} & r\mathbbm{1}_k^\top M_{12}^{(T)} + z^\top M_{22}^{(T)}\end{array}\right] \\
2m w_Q^\top M_Q & = & \left[\begin{array}{c|c} kr & z^\top\end{array}\right] \left[\begin{array}{c|c} m_{11}^{(Q)} & M_{12}^{(Q)} \\ \hline M_{21}^{(Q)} & M_{22}^{(Q)}\end{array}\right]\\
& = & \left[\begin{array}{c|c} kr m_{11}^{(Q)} + z^\top M_{21}^{(Q)} & kd M_{12}^{(Q)} + z^\top M_{22}^{(Q)}\end{array}\right]
\end{eqnarray*}
We want to show that 
\[r\mathbbm{1}_k^\top M_{12}^{(T)} + z^\top M_{22}^{(T)} = kr M_{12}^{(Q)} + z^\top M_{22}^{(Q)}.\]

The stochastic complements of $T_{22}$ and $Q_{22}$ agree, since 
\begin{eqnarray*}
S_2^{(T)} & = & T_{22} + T_{21}T_{12} 
=\begin{bmatrix} A & B \\ C & D\end{bmatrix} + \left[\begin{array}{c} \mathbf{x}\mathbbm{1}_k^\top \\ \hline 0 \end{array} \right]\left[\begin{array}{c|c} \frac{1}{r}\mathbbm{1}_k\mathbbm{1}_r^\top & 0 \end{array}\right]\\
& = & \begin{bmatrix} A + \frac{k}{r}\mathbf{x}\mathbbm{1}_r^\top & B \\ C & D\end{bmatrix}
=S_2^{(Q)}.
\end{eqnarray*}
Hence $M_{S_2}^{(T)} = M_{S_2}^{(Q)}.$

Next, consider that $M_{22}^{(T)} = \gamma_2M_{S_2}^{(T)} + V_2^{(T)}$ as described in Appendix~\ref{app:stoch_comp}, and an analogous expression holds for $M_{22}^{(Q)}$. Without directly computing, we can show that these submatrices are equal. Note that
\begin{eqnarray*}
V_2^{(T)} & = & (I-S_2^{(T)})^\# \begin{bmatrix} \mathbf{x}\mathbbm{1}_k^\top \\ 0 \end{bmatrix}\mathbbm{1}_k\mathbbm{1}_{n-k}^\top - \left((I-S_2^{(T)})^\# \begin{bmatrix} \mathbf{x}\mathbbm{1}_k^\top \\ 0 \end{bmatrix}\mathbbm{1}_k\mathbbm{1}_{n-k}^\top\right)^\top\\
& = & (I-S_2^{(T)})^\# \begin{bmatrix} k\mathbf{x} \\ 0 \end{bmatrix}\mathbbm{1}_{n-k}^\top - \left((I-S_2^{(T)})^\# \begin{bmatrix} k\mathbf{x} \\ 0 \end{bmatrix}\mathbbm{1}_{n-k}^\top\right)^\top.
\end{eqnarray*}
But $S_2^{(T)} = S_2^{(Q)}$, and $\begin{bmatrix} k\mathbf{x} \\ 0 \end{bmatrix} = Q_{21}$, so this is in fact the expression for $V_2^{(Q)}$ also. Hence $M_{22}^{(T)} = M_{22}^{(Q)}$ (since  $\gamma_2^{(Q)} = \gamma_2^{(T)} = \frac{1}{\mathbbm{1}^\top w_2}$).
Thus $w_2^\top M_{22}^{(T)} = w_2^\top M_{22}^{(Q)}$. It remains to show that \[r\mathbbm{1}_k^\top M_{12}^{(T)} = kr M_{12}^{(Q)}.\]
Note that
\begin{eqnarray*}
M_{12}^{(T)} & = & (I-T_{11})^{-1}T_{12}\left( M_{22}^{(T)} - (M_{22}^{(T)})_{dg}\right) + (I-T_{11})^{-1}J_{k,n-k}\\
& = & \tfrac{1}{r}\mathbbm{1}_k \begin{bmatrix} \mathbbm{1}_r^\top \mid 0^\top \end{bmatrix} \left( M_{22}^{(T)} - (M_{22}^{(T)})_{dg}\right) + \mathbbm{1}_k\mathbbm{1}_{n-k}^\top,
\end{eqnarray*}
and hence
\begin{eqnarray*}
r\mathbbm{1}_k^\top M_{12}^{(T)} & = & \mathbbm{1}_k^\top\mathbbm{1}_k \begin{bmatrix} \mathbbm{1}_r^\top \mid 0^\top \end{bmatrix} \left( M_{22}^{(T)} - (M_{22}^{(T)})_{dg}\right) + r\mathbbm{1}_k^\top\mathbbm{1}_k\mathbbm{1}_{n-k}^\top\\
& = & k\begin{bmatrix} \mathbbm{1}_r^\top \mid 0^\top \end{bmatrix} \left( M_{22}^{(T)} - (M_{22}^{(T)})_{dg}\right) + rk\mathbbm{1}_{n-k}^\top.
\end{eqnarray*}
But %$M_{12}^{(Q)}$ 
by definition, % is 
\begin{eqnarray*}
M_{12}^{(Q)} & = & (I-Q_{11})^{-1}Q_{12}\left( M_{22}^{(Q)} - (M_{22}^{(Q)})_{dg}\right) + (I-Q_{11})^{-1}\mathbbm{1}_{n-k}^\top\\
& = & \tfrac{1}{r}\begin{bmatrix} \mathbbm{1}_r^\top \mid 0^\top \end{bmatrix} \left( M_{22}^{(Q)} - (M_{22}^{(Q)})_{dg}\right) + \mathbbm{1}_{n-k}^\top,
\end{eqnarray*}
and since $M_{22}^{(Q)} = M_{22}^{(T)}$,  \quad 
$r\mathbbm{1}_k^\top M_{12}^{(T)} = kr M_{12}^{(Q)}.$

Therefore, in a graph $G$ with an independent set $U$ of twin vertices,  for all $u\in V(G)\backslash U$, \quad 
$\alpha_T(u) = \alpha_Q(u).$

\item  By Equation~(\ref{eq:alphaw}), $\mathcal{K}(T) = w_T^\top \alpha_T$, $\mathcal{K}(Q) = w_Q^\top \alpha_Q$, and further 
\[w_T^\top = \frac{1}{2m}\begin{bmatrix} r \mathbbm{1}_k^\top \mid z^\top \end{bmatrix}, \quad w_Q\top = \frac{1}{2m}\begin{bmatrix} kr \mid z^\top \end{bmatrix}.\]
By Lemma~\ref{lem:kemeny_quotient},  $\mathcal{K}(T) = \mathcal{K}(Q)+(k-1).$
Using the substitutions above for $\mathcal{K}(T)$ and $\mathcal{K}(Q)$ in this equation, along with using the conclusions of parts (a) and (b) of this proposition to reduce the resulting equation,  for any $u\in U$, 
\begin{eqnarray*}
\frac{kr}{2m}\alpha_T(u) & = & \frac{kr}{2m}\alpha_Q(v_0) + (k-1)\\
\alpha_T(u) & = & \alpha_Q(v_0) + \frac{2m(k-1)}{kr}.
\end{eqnarray*}
\end{enumerate}
\end{proof}

When a graph has a set of twin vertices, the quotient matrix obtained from the equitable partition of the transition matrix will correspond to a random walk on a weighted graph, where the set of twin vertices has been coalesced to a single vertex. As such, the modification to add all edges between vertices in the independent set of twins corresponds to adding a weighted loop to the single coalesced vertex. In the quotient transition matrix, this corresponds to a rank-one update. 

\begin{theorem}\cite{hu2019complete}\label{thm:rkone}
Let $T$ be an irreducible transition matrix, and let $\tilde{T} = T+e_iu^\top$ be a perturbation of the $i^{th}$ row of $T$, such that $\tilde{T}$ is irreducible and stochastic (i.e. $u^\top\mathbbm{1} = 0$). Fix $k\neq i$, and let 
\[\theta_j = \frac{u^\top(I-T)^\#e_j}{1-u^\top(I-T)^\#e_i}.\]
If $T$ has stationary vector $w$ and mean first passage matrix $M$, then 
\[\mathcal{K}(\tilde{T}) - \mathcal{K}(T) = w_i\sum_j \theta_j (m_{k,i} - m_{j,i}) + \theta_i.\]
\end{theorem}
We are now ready to state and prove the following result determining the difference in Kemeny's constant when adding a clique to an independent set of twin vertices.

\begin{theorem}\label{thm:braess_clique_diff_acc}
Let $G$ be a graph of order $n$ and size $m$ with an independent set of twin vertices of order $k$, and let $r$ be the common degree of these vertices. Let $T$ be the transition matrix for the random walk on $G$. Let $\alpha_1$ be the accessibility index of one of the twin vertices. Let $\tilde{T}$ be the transition matrix obtained when inserting all possible edges between the twin vertices. 
Then 
\[\mathcal{K}(\tilde{T}) - \mathcal{K}(T) = \frac{k(k-1)}{2m+k(k-1)}\left(\alpha_1 - \frac{2m(k-1)}{kr}\right) - \frac{k-1}{r+k}.\]
\end{theorem}

\begin{proof}
Let $G$ be a graph of order $n$, and let $T$ be the transition matrix for the random walk on $G$, written in the form \eqref{eq:T_ind_set}, and $Q$ its quotient matrix as described in \eqref{eq:Q_ind_set}.

By Lemma \ref{lem:kemeny_quotient},
\begin{equation}\label{eq:diffKT}\mathcal{K}(\tilde{T}) - \mathcal{K}(T) = \mathcal{K}(\tilde{Q}) - \mathcal{K}(Q) - \frac{k-1}{r+k}.\end{equation}
To analyse the difference $\mathcal{K}(\tilde{Q}) - \mathcal{K}(Q)$, note that $\tilde{Q}$ can be written as a rank-one update of $Q$. In particular,  
$\tilde{Q}=Q+e_1u^\top$
with 
\begin{eqnarray*}u^\top & = & \left[\begin{array}{c|c|c} \frac{k-1}{r+k-1} & -\frac{(k-1)}{r(r+k-1)}\mathbbm{1}_r^\top & 0_{n-r-k}^\top\end{array}\right] \\
& = & \frac{k-1}{r+k-1} \left[\begin{array}{c|c|c} 1 & \tfrac{-1}{r}\mathbbm{1}_r^\top & 0^\top_{n-r-k}\end{array}\right]\\
& = & \frac{k-1}{r+k-1} {e}_1^\top(I-Q).
\end{eqnarray*}
We substitute this into the expression $\theta_j = \frac{u^\top(I-Q)^\#e_j}{1-u^\top(I-Q)^\#e_1}$ in Theorem~\ref{thm:rkone} and simplify using properties of the group inverse; namely, the fact that $(I-Q)(I-Q)^\# = I -\mathbbm{1}w^\top$ with $w=w_Q^\top$. Thus
\[\theta_j = \left\{\begin{array}{cc} \dfrac{-(k-1)w_j}{r+(k-1)w_1}, & \mbox{if $j\neq 1$;}\\[10pt]
\dfrac{(k-1)(1-w_1)}{r+(k-1)w_1}, & \mbox{if $j=1$.}\end{array}\right.\]
%From this, we can conclude that 
It follows that $\sum_{j=1}^{n-k+1} \theta_j = 0$. Therefore, the expression $\mathcal{K}(\tilde{Q})-\mathcal{K}(Q)$ simplifies substantially. First, fix any $\ell\neq 1$. By Theorem~\ref{thm:rkone},
\begin{eqnarray*}
\mathcal{K}(\tilde{Q})-\mathcal{K}(Q) & = & w_1 \sum_{j=1}^{n-k+1} \theta_j(m_{\ell, 1} - m_{j,1}) + \theta_1 \\
&= & w_1m_{\ell,1}\sum_j \theta_j -w_1\sum_j \theta_jm_{j,1} + \theta_1 
\\
&=& 
-w_1\sum_j\theta_jm_{j,1} + \theta_1 
=-w_1\sum_{j\neq 1}\theta_jm_{j,1} -w_1\theta_1m_{1,1} + \theta_1 \\
& = & -w_1\sum_{j\neq 1}\theta_jm_{j,1} -w_1\theta_1 (\tfrac{1}{w_1}) + \theta_1 
=-w_1\sum_{j\neq 1}\theta_jm_{j,1}\\
& = & \dfrac{(k-1)w_1}{r+(k-1)w_1}\sum_{j\neq 1}w_jm_{j,1}.
\end{eqnarray*}
The final summation is the accessibility index for the first state in the quotient transition matrix $Q$. Thus by Proposition~\ref{prop:acc_quotient} (c),  this summation is $\alpha_1 - \frac{2m(k-1)}{kr}$.
Hence, by (\ref{eq:diffKT}),
\begin{eqnarray*}
\mathcal{K}(\tilde{T})-\mathcal{K}(T) & = & \mathcal{K}(\tilde{Q}) - \mathcal{K}(Q) - \frac{k-1}{r+k}\\
& = & \dfrac{(k-1)w_1}{r+(k-1)w_1}\left(\alpha_1- \frac{2m(k-1)}{kr}\right) - \dfrac{k-1}{r+k}.\end{eqnarray*}
Finally, since $w_1= \frac{kr}{2m}$, \begin{eqnarray*}
\mathcal{K}(\tilde{T})-\mathcal{K}(T) 
& = & \dfrac{k(k-1)}{2m+k(k-1)}\left(\alpha_1- \frac{2m(k-1)}{kr}\right) - \dfrac{k-1}{r+k}.
\end{eqnarray*}

\end{proof}

\begin{remark}
We note that the approach in this section borrows from the one given in \cite{hu2019complete} for the change in Kemeny's constant when adding a single edge between two non-adjacent twin vertices. In particular, the proof makes use of equitable partitions in the transition matrix in order to reduce the problem. However, we go further with rewriting the final statement of the result to clarify  the dependence of the change in Kemeny's constant on key parameters such as the accessibility index, the size of the graph, as well as the order and common degree of the independent set.
\end{remark}

\section{Existence of Braess cliques in certain graph families}

\subsection{Graphs with twin pendent vertices}
In this section, we apply the general results regarding equitable partitions of the previous section to graphs which have $k$ pendent vertices attached to some specified vertex $v$. Our goal in this section is to determine the conditions under which the independent set created by the addition of $k$ pendent vertices to a vertex in a graph forms a Braess clique. Given the statement of Theorem~\ref{thm:braess_clique_diff_acc}, this requires us to examine how the accessibility indices of different vertices in the graph $G$ change upon the addition of $k$ pendent vertices to form $\widehat{G}$.

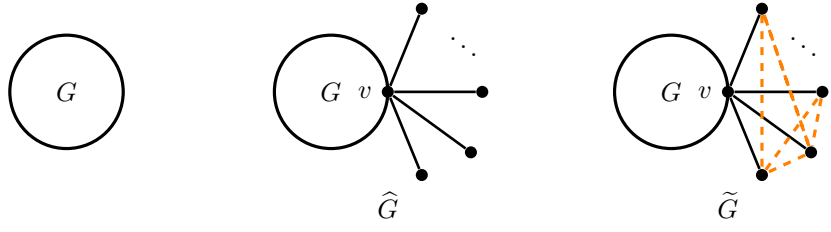
\begin{figure}
\begin{center}
\begin{tikzpicture}[vtx/.style={circle,fill,inner sep=1.6pt},
    bubble/.style={draw,very thick,circle,minimum size=1.5cm},
    ed/.style={line width=1pt},]

\begin{scope}
  \node[bubble] (B0) at (0,0) {};
  \node at (B0.center) {$G$};
\end{scope}

\begin{scope}[xshift=3.5cm]
  % bubble and attachment vertex
  \node[bubble] (B1) at (0,0) {};
  \node at (B1.center) {$G$};
  \node[vtx,label=left:$v$] (a) at (0.75,0) {}; % on the right side of the bubble

  % pendents
  \node[vtx] (b) at (2,0) {};
  \node[vtx] (c) at (1.2,1.1) {};
  \node[vtx] (d) at (1.2, -1.1) {};
  \node[vtx] (e) at (1.85, -0.8) {};

  \node[draw=none,fill=none] at (1.75,0.7) {$\ddots$};

  \draw[ed] (a)--(b);
  \draw[ed] (a)--(c);
  \draw[ed] (a)--(d);
  \draw[ed] (a)--(e);

  \node at (0.75,-1.5) {$%G^*
  \widehat{G}$};
\end{scope}

\begin{scope}[xshift=8.0cm]
  % bubble and base path
  \node[bubble] (B2) at (0,0) {};
  \node at (B2.center) {$G$};
  \node[vtx,label=left:$v$] (a2) at (0.75,0) {};
  \node[vtx] (b2) at (2,0) {};
  \node[vtx] (c2) at (1.2,1.1) {};
  \node[vtx] (d2) at (1.2,-1.1) {};
  \node[vtx] (e2) at (1.85, -0.8) {};
  \node[draw=none,fill=none] at (1.75,0.7) {$\ddots$};

  \draw[ed] (a2)--(b2);
  \draw[ed] (a2)--(c2);
  \draw[ed] (a2)--(d2);
  \draw[ed] (a2)--(e2);

  % orange Kl among the pendent vertices
  \draw[very thick,orange, dashed] (c2)--(e2)--(d2)--(b2)--(e2)--(c2)--(d2);

  \node at (0.75,-1.5) {$\widetilde{G}%G^\circledast
  $};
\end{scope}
\end{tikzpicture}
\end{center}
\caption{A graph $G$ with $k$ twin pendent vertices attached to form $\widehat{G}$. Corollary~\ref{cor:pendent} provides a condition for when the addition of a clique to form $\widetilde{G}$ causes an increase in Kemeny's constant.}
\label{fig:k_pendent}
\end{figure}

The following result is well-known (see \cite[Theorem 6.2.1]{stevesbook}), and is referred to as a triangle inequality for mean first passage times.
\begin{lemma}\label{lem:mfp_triangle}
Let $T$ be an irreducible stochastic matrix, and denote the corresponding mean first passage matrix by $M$. Then for any triple of indices $i,j, k$,   
$m_{i,j} + m_{j,k} \geq m_{i,k},$
with equality if and only if $j$ is distinct from both $i$ and $k$, and in addition, every path from vertex $i$ to vertex $k$ passes through vertex $j$. 
\end{lemma}

\begin{lemma}
Let $G$ be a graph with $m$ edges and let $v\in V(G)$. If $\widehat{G}$ is the graph formed by adding $k$ pendent vertices $\{v_1, v_2, \ldots, v_k\}$ to $v$ (as in Figure~\ref{fig:k_pendent}), then 
\[\alpha_{\widehat{G}}(v_i) = \frac{2m}{2m+2k}\alpha_G(v) + \frac{(2m+2k-1)^2}{2m+2k} + \frac{k-1}{2m+2k}.\]
\end{lemma}

\begin{proof}
Let $G$ be a graph with $m$ edges, and $\widehat{G}$ the graph formed by adding $k$ pendents  $\{v_1, v_2, \ldots, v_k\}$ to some $v \in V(G)$. Then 
\begin{eqnarray*} 
\alpha_{\widehat{G}}(v_1) & = & \sum_{\substack{u\in V(\widehat{G})\\u\neq v_1}} \widehat{w}_u \widehat{m}_{u, v_1}  =  \sum_{\substack{u\in V(\widehat{G})\\u\neq v_1}} \frac{\deg_{\widehat{G}}(u)}{2m+2k} \widehat{m}_{u, v_1} \\
& = & \sum_{\substack{u\in V(G)\\u\neq v}} \frac{\deg_{\widehat{G}}(u)}{2m+2k} \widehat{m}_{u, v_1} + \sum_{i=2}^k \frac{1}{2m+2k} \widehat{m}_{v_i, v_1} + \frac{\deg_{\widehat{G}}(v)}{2m+2k} \widehat{m}_{v, v_1}.
\end{eqnarray*}
We now rewrite as many of these terms as possible in terms of quantities for the random walk on $G$. Note that for $u\in V(G)$, $u\neq v$, $\deg_{\widehat{G}}(u) = \deg_G(u)$, and since every walk from $u$ to $v_1$ in $\widehat{G}$ must pass through $v$, by Lemma~\ref{lem:mfp_triangle}, $\widehat{m}_{u, v_1} = m_{u, v} + \widehat{m}_{v, v_1}.$ Similarly, for $i\neq 1$, 
\begin{equation}\label{eq:onup}\widehat{m}_{v_i, v_1} = \widehat{m}_{v_i, v} + \widehat{m}_{v, v_1} = 1 + \widehat{m}_{v, v_1}.\end{equation}
Finally, it is clear that $\deg_{\widehat{G}}(v) = \deg_{G}(v) + k$. Hence
\begin{eqnarray*}
\alpha_{\widehat{G}}(v_1) & = &  \sum_{\substack{u\in V(G)\\u\neq v}} \frac{\deg_{G}(u)}{2m+2k} (m_{u, v} + \widehat{m}_{v, v_1} ) + \sum_{i=2}^k \frac{1}{2m+2k} \left(1+ \widehat{m}_{v, v_1}\right)+ \frac{\deg_{\widehat{G}}(v)}{2m+2k} \widehat{m}_{v, v_1}\\
& = & \sum_{\substack{u\in V(G)\\u\neq v}} \frac{\deg_{G}(u)}{2m+2k}m_{u, v} + \frac{\widehat{m}_{v, v_1}}{2m+2k}\left( \sum_{\substack{u\in V(G)\\u\neq v}} \deg_G(u) + (k-1) + \deg_G(v) + k\right) + \frac{k-1}{2m+2k}\\
& = & \frac{2m}{2m+2k}\sum_{\substack{u\in V(G)\\u\neq v}} \frac{\deg_{G}(u)}{2m}m_{u, v} + \frac{2m+2k-1}{2m+2k} \widehat{m}_{v, v_1} + \frac{k-1}{2m+2k}.
\end{eqnarray*}
It remains to determine the value of $\widehat{m}_{v, v_1}$. By Lemma~\ref{lem:mfp_triangle}, $\widehat{m}_{v_1, v_1} = 
%& = & 
\widehat{m}_{v_1, v} + \widehat{m}_{v, v_1}$. %\\
Since in general $m_{i,i}=\frac{1}{w_i}$,  using Equation~(\ref{eq:onup}), 
\begin{eqnarray*}
\frac{1}{\left( \frac{\deg_{\widehat{G}}(v_1)}{2m+2k}\right)} & = & 1 + \widehat{m}_{v, v_1} \qquad {\rm{\ so\ that\ \qquad }}
\widehat{m}_{v, v_1}  =  2m+2k - 1.
\end{eqnarray*}
Therefore
\[\alpha_{\widehat{G}}(v_i) = \frac{2m}{2m+2k}\alpha_G(v) + \frac{(2m+2k-1)^2}{2m+2k} + \frac{k-1}{2m+2k}.\]
\end{proof}

\begin{lemma}
Let $G$ be a graph with $m$ edges and let $v\in V(G)$. If $\widehat{G}$ is the graph formed by adding $k$ pendent vertices to $v$, then 
\[\alpha_{\widehat{G}}(v) = \frac{2m}{2m+2k} \alpha_G(v) + \frac{k}{2m+2k}.\]
\end{lemma}

\begin{proof}
Let $G$ be a graph with $m$ edges and with some vertex $v$. Let $v_1, v_2, \ldots, v_k$ denote the $k$ pendent vertices adjacent to $v$ in $\widehat{G}$. Then 
\begin{eqnarray*}
\alpha_{\widehat{G}}(v) & = & \sum_{\substack{u\in V(\widehat{G})\\u\neq v}} \frac{\deg_{\widehat{G}}(u)}{2m+2k} \widehat{m}_{u, v} %\\
%& = & 
= \sum_{\substack{u\in V(G)\\ u\neq v}} \frac{\deg_{\widehat{G}}(u)}{2m+2k} \widehat{m}_{u, v} + \sum_{i=1}^k \frac{1}{2m+2k} \widehat{m}_{v_i, v}\\
& = & \frac{2m}{2m+2k} \sum_{\substack{u\in V(G)\\ u\neq v}} \frac{\deg_{G}(u)}{2m} m_{u, v} + \sum_{i=1}^k \frac{1}{2m+2k} \widehat{m}_{v_i, v} %\\
%& = & 
= \frac{2m}{2m+2k} \alpha_G(v) + \frac{k}{2m+2k},
\end{eqnarray*}
where the last line follows by definition of the accessibility index, and the fact that $\widehat{m}_{v_i, v} = 1$ for each $i\in\{ 1, \ldots, k\}$.
\end{proof}

\begin{theorem}\label{thm:kem_diff_k_pendent}
Let $G$ be a graph with $m$ edges, and let $v \in V(G)$. Let $\alpha_G(v)$ denote the accessibility index of $v$ in $G$. 
Let $\widehat{G}$ be the graph obtained by gluing $k$ pendent vertices to $v$, and let $T$ be the transition matrix for the random walk on $\widehat{G}$. Let $\tilde{G}$ be the graph obtained from $\widehat{G}$ by adding all possible edges between the $k$ pendent vertices, and $\tilde{T}$ the transition matrix for the random walk on $\tilde{G}$. Then 
\begin{eqnarray*}
\mathcal{K}(\tilde{T}) - \mathcal{K}(T) & = & \frac{k(k-1)}{2(m+k)+k(k-1)}\left(\frac{m}{m+k}\alpha_G(v) + \frac{k-1}{2(m+k)} + \frac{(2m+2k-1)^2}{2(m+k)}\right) \\
& & \qquad \qquad  - \frac{2(m+k)(k-1)^2}{2(m+k)+k(k-1)} - \frac{k-1}{k+1}.\end{eqnarray*}
\end{theorem}

\begin{corollary}\label{cor:pendent}
Let $G$ be a connected graph with $m$ edges and $v\in V(G)$. 
Let $\widehat{G}$ be the graph obtained by attaching $k$ pendent vertices to $v$. A sufficient condition for the resulting independent set to form a Braess clique is that 
\[m > \frac{k+1}{4}.\]
\end{corollary}

\begin{proof}
Suppose that $\mathcal{K}(\tilde{T}) - \mathcal{K}(T)\geq 0$. Then rearranging the expression in Theorem~\ref{thm:kem_diff_k_pendent} for $\alpha_G(v)$ gives
\[ 
\alpha_G(v)  \geq  \frac{k^2-2km-4m^2+k+2m}{2m(k+1)}.
\] 
 Note that $\alpha_G(v) \geq \frac12$ by Lemma~\ref{lem:alpha}.
Note that 
 \[\frac{k^2-2km-4m^2+k+2m}{2m(k+1)} \leq \frac{1}{2}\]
 if and only if
 $4m^2 + (3k-1)m-(k^2+k)\geq 0.$
 The roots of this quadratic are $m=\frac14(k+1)$ and $m=-k$. Thus a sufficient condition for the addition of $k$ pendent twins to form a Braess clique is that 
$m>\frac{k+1}{4}.$
\end{proof}

Given a graph $G$ that has a set of $k$ pendent vertices, Theorem~\ref{thm:braess_clique_diff_acc} already provides an expression for the change in Kemeny's constant upon addition of the edges of a clique. This expression can be analyzed to determine when this clique is Braess. However, our goal for this section was to consider a graph operation by which we could create the circumstances and structure of a Braess clique in a given graph. Any restrictions on the given graph would have then given some insight into the intuition governing the existence of a Braess clique in a graph. Somewhat surprisingly, the sufficient condition we produce is not very restrictive. Adding multiple edges between pendent twins was considered in \cite[Section 4.1]{faught20221}, without a specific  focus on cliques.

We also note 
that our results in this section extend the results of \cite{ciardo2020braess}. 
The proof of \cite[Theorem 2.2]{ciardo2020braess} and the results in \cite{faught20221} use techniques from resistance distances in graphs; our results focus on the accessibility index. 

In \cite{kirkland2016kemeny} it was demonstrated that almost all trees have a Braess edge.
In \cite{ciardo2020braess}, it was demonstrated that almost every connected planar labeled graph has a Braess edge. The argument can be extended to a Braess $K_\ell$ for $\ell\geq 2$ using the results of this section by simply replacing the graph $P_3$ in the argument in \cite{ciardo2020braess} with a star graph on $\ell+1$ vertices.  
\begin{corollary}
    Given $\ell\geq 2,$ almost every connected planar labeled graph has
    a Braess $K_\ell$.
\end{corollary}

\subsection{Complete bipartite graphs}
In  \cite{hu2019complete}, it was demonstrated that a complete bipartite graph $K_{\ell,n-\ell}$ has a Braess edge ($K_2$) in the part of order 
$\ell$ if and only if $n<\frac{7\ell-8}{3}$.
In this section, we extend  this result to Braess cliques.

It is known (see e.g. \cite{G2016}) that for a complete bipartite graph $K_{a,b}$ having parts
$A$ and $B$ with $|A|=a$ and $|B|=b,$ the resistance distance between a pair of vertices is 
\begin{equation*}
r(x,y)=\left\{ 
\begin{array}{cl}
\frac{2}{a}& \text{if } x,y\in B\\[6pt]
\frac{2}{b}& \text{if } x,y\in A\\[6pt]
\frac{a+b-1}{ab}& \text{if } x \text{ is adjacent to } y.\\
\end{array}
\right.
\end{equation*}
Thus for $v\in A$,
\begin{eqnarray*}
\mu_{K_{a,b}}(v)&=&\sum_{\substack{i\in A\\i\neq v}} d_i r(i,v) +\sum_{i\in B} d_i r(i,v)
= (a-1)b\left(\frac{2}{b}\right)+ba\left(\frac{a+b-1}{ab}\right)=3a+b-3.
\end{eqnarray*}
Likewise, for $v\in B$, $\mu_{K_{a,b}}(v)=3b+a-3.$ Further, as noted in \cite{KD}, 
\begin{equation}\label{eq:Kbip}
\mathcal{K}(K_{a,b})=a+b-\frac{3}{2}.
\end{equation}
Therefore, using~(\ref{eq:amu}),
\begin{equation}\label{eq:alphabip}
\alpha_{K_{a,b}}(v) = \left\{
\begin{array}{cl}
\frac{4a-3}{2} & \text{if } v\in A \\[6pt]
\frac{4b-3}{2} & \text{if } v\in B.
\end{array}
\right.
\end{equation}
\begin{theorem}\label{thm:bipart}
Let $G$ be the complete bipartite graph $K_{\ell, n-\ell}$, $2\leq \ell \leq n-1$, and let $2\leq k\leq \ell$. Then $G$ has a Braess clique of order $k$ in the part of order $\ell$ if and only if 
\[k< \frac{7\ell-3n+2}{5}.\]
\end{theorem}

\begin{proof}
Let $A$ and $B$ denote the two partite sets of $K_{\ell, n-\ell}$ of orders $\ell$ and $n-\ell$, respectively. 
By (\ref{eq:alphabip}), 
 $\alpha(v) = \frac{4l-3}{2}$ for $v\in A$.  The common degree of vertices in the set $A$ is $n-\ell$. Thus by
 Theorem~\ref{thm:braess_clique_diff_acc},  
\begin{eqnarray*}
\mathcal{K}(\tilde{T}) - \mathcal{K}(T) & = & \frac{k(k-1)}{2\ell(n-\ell)+k(k-1)}\left(\frac{4\ell-3}{2} - \frac{2\ell(n-\ell)(k-1)}{k(n-\ell)}\right) - \frac{k-1}{n-\ell+k}\\
& = & \frac{-k(k-1)(5k-7\ell+3n-2)}{2(k(k-1)+2\ell(n-\ell))(n+k-1)}.
\end{eqnarray*}
This rational expression is positive 
exactly when the factor   
$-(5k-7\ell +3n-2)$ is positive. 
That is, there is an increase in Kemeny's constant if and only if 
$5k < {7\ell-3n+2}.$
\end{proof}

One way to understand the known results regarding Braess edges in \cite{hu2019complete} is by observing that $K_{\ell, n-\ell}$ has a Braess edge in the part of order $\ell$ if and only if \[\ell>\tfrac{3}{4}(n-\ell) + 2;\]
that is, when the part in which the edge is placed is at least a little larger than three-quarters the size of the other part. If one of the partite sets is too small compared to the other, then that set does not contain a Braess edge.

Similarly, the conclusion of Theorem \ref{thm:bipart} may be investigated further by rewriting the condition in several ways. In particular, the complete bipartite graph $K_{\ell, n-\ell}$ has a Braess clique of order $k\leq \ell$ in the part of order $\ell$ if and only if 
\[k \leq \tfrac45 \ell -\tfrac35(n-\ell) + \tfrac25.\]
Alternatively, the complete bipartite graph $K_{a, b}$ has a Braess clique of order $k\leq a$ in the part of order $a$ if and only if 
\[k \leq \tfrac45 a -\tfrac35 b + \tfrac25.\]
If $b=ra$ (that is, $b$ is some linear multiple of $a$),  $K_{a, ra}$ has a Braess clique of order $k\leq a$ in the part of order $a$ if and only if $k\leq \frac{4-3r}{5} a + \frac25$. If $b$ is approximately $a$, then a Braess clique can take up approximately one-fifth of the vertices in one partite set. If $b$ is significantly larger than $a$, then there is no Braess clique on the smaller side. If $b$ is smaller than $a$, the maximum size of a clique which is Braess is larger (for example, if $b=\frac12 a$, up to approximately half of the vertices in the larger side create a Braess clique). 

\begin{example}
While the graph $G=K_{90,10}$ has many Braess edges, it also has many Braess cliques. In this graph $G$, for cliques up to order 32, inserting a larger clique gives a larger increase in Kemeny's constant. In some sense, we could say the larger cliques are more Braess than the small cliques up to a maximum.
In particular, the expression for the difference in Kemeny's constant in the proof of Theorem~\ref{thm:bipart} can be used to show that inserting a clique of order 33 in $G$ will have the largest increase in Kemeny's constant. See Figure~\ref{fig:cb_clique} for a plot of the difference in Kemeny's constant when inserting a clique of order $k$, $2\leq k\leq \ell,$ into 
$K_{\ell, 100-\ell}$ for $\ell=75, 80, 85, 90$.
\end{example}
\begin{figure}[h]
    \begin{center}
    \includegraphics[width=5in]{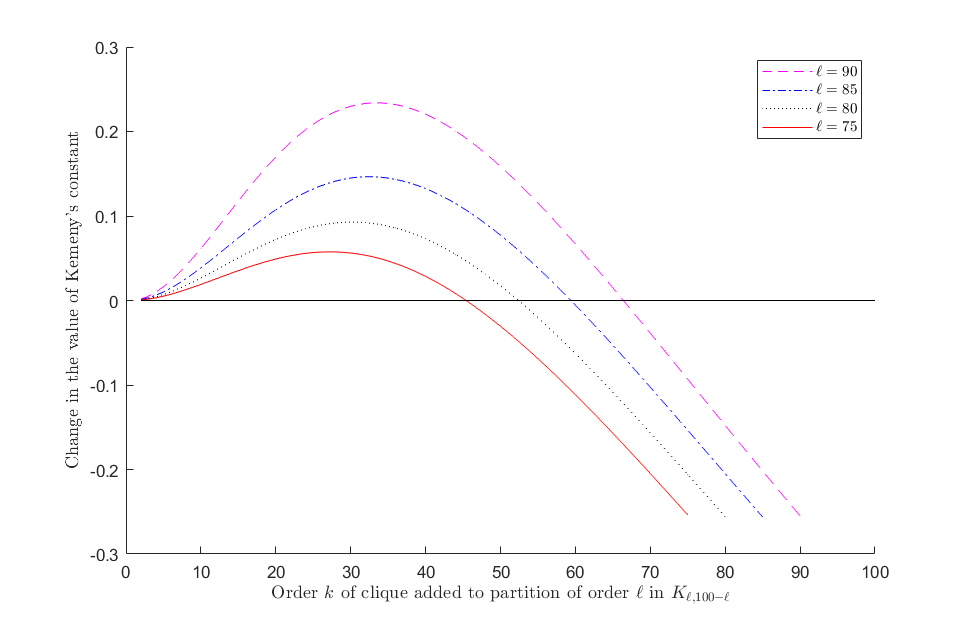}
    \end{center}
    \caption{Difference in Kemeny's constant when adding different size cliques to $K_{\ell, 100-\ell}$}\label{fig:cb_clique}
\end{figure}

\section{Interplay of Braess edges and Braess cliques}

In the previous section, we observed the existence of Braess cliques of size $\ell$ for some $\ell>2$. In the case of bipartite graphs, Theorem~\ref{thm:bipart}, and also
in the graphs with pendent twins, Corollary~\ref{cor:pendent}, the Braess cliques were composed of edges, each of which is a Braess edge in the original graph. This prompts natural questions about the composition of Braess cliques  cliques of order $\ell>2$. For example, is this a method for locating Braess cliques of order $\ell>2$? In this section we observe that not every clique of Braess edges will be a Braess clique. We will also see that a Braess clique 
could include some edges that are not Braess edges. In fact, there are also Braess cliques which have no edges that are Braess edges. 
The goal we pursue is to further our understanding of the behaviour of Kemeny's constant of a graph, and when the addition of edges or cliques causes counterintuitive behaviour, as well as determining the features in such graphs on which this behaviour depends.

In the next theorem, we demonstrate that %\begin{remark}
    a coclique  of $G$ consisting of Braess edges is not necessarily a Braess clique for $G$.
%\end{remark}

\begin{theorem}\label{thm: K2,m}
For $\ell \geq 4$, the complete bipartite graph $K_{\ell,2}$ has a coclique consisting of Braess edges between the $\ell$ vertices of degree $2$, but that coclique is not a Braess clique of $K_{\ell,2}$. 
\end{theorem}
\begin{proof}
By Theorem~\ref{thm:bipart},  $K_{\ell,2}$ will have a Braess edge  incident to a pair of vertices of degree $2$ if and only if 
$\ell\geq 4.$
Similarly,  $K_{\ell,2}$ has a Braess $K_\ell$ if and only if 
$\ell \leq -4$. 
\end{proof}

%\vspace{2cm}
\begin{remark}
    Not every Braess clique contains a Braess edge (see Figure~\ref{fig:cliquenoedge}). The graph in Figure~\ref{fig:cliquenoedge}
belongs to a  family of graphs $H$ that have a Braess clique, none of whose edges are Braess edges for $H$, as demonstrated in Theorem~\ref{thm:spider}.  
\end{remark}

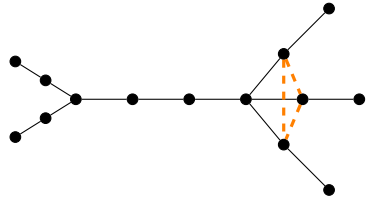
\begin{figure}[h] 
\begin{center}\begin{tikzpicture}[scale=0.5, every node/.style={circle, fill,draw, minimum size=4pt, inner sep=0pt}]

% Main path on 4 vertices
\node (v1) at (0,0) {};
\node (v2) at (1.5,0) {};
\node (v3) at (3.0,0) {};
\node (v4) at (4.5,0) {};

\draw (v1)--(v2)--(v3)--(v4);

% Left broom (2 legs attached to v1, each length 2)
\node (l1a) at (-0.8,  0.5) {};
\node (l1b) at (-1.6,  1) {};
\draw (v1)--(l1a)--(l1b);

\node (l2a) at (-0.8, -0.5) {};
\node (l2b) at (-1.6, -1) {};
\draw (v1)--(l2a)--(l2b);

% Right broom (3 legs attached to v4, each length 2)
\node (r1a) at (5.5,  1.2) {};
\node (r1b) at (6.7,  2.4) {};
\draw (v4)--(r1a)--(r1b);

\node (r2a) at (6,  0.0) {};
\node (r2b) at (7.5,  0.0) {};
\draw (v4)--(r2a)--(r2b);

\node (r3a) at (5.5, -1.2) {};
\node (r3b) at (6.7, -2.4) {};
\draw (v4)--(r3a)--(r3b);

% Braess K3 edges (between the 3 inner right broom vertices)
\draw[very thick,orange, dashed] (r1a)--(r2a);
\draw[very thick,orange, dashed] (r2a)--(r3a);
\draw[very thick,orange, dashed] (r1a)--(r3a);

\end{tikzpicture}
\end{center}
\caption{A graph having a Braess $K_3$ composed of edges that are not Braess edges.}\label{fig:cliquenoedge}
\end{figure}
%\end{example}

 A \emph{spider graph} is a tree with one vertex of degree at least 3, called the \emph{center vertex}, and all other vertices of degree one or two. Given $a\geq 1$ and $b\geq 3$, let $S_{a,b}$ denote the spider graph with center vertex of degree $b$ and with each pendent vertex at distance $a$ from the center. It was observed in \cite{jang2025kemeny} 
that $\K (S_{a,b})=a^2\left(b-{\frac{2}{3}}\right)+\frac{1}{6}$.

\begin{theorem}\label{thm:spider}
    Let $G$ be a graph and $v\in V(G)$ with $\alpha=\alpha_{G}(v)$ and $m=|E(G)|$. Let $H$ be the graph obtained by attaching  the center of a spider graph $S_{2,3}$ to $G$ at $v$, as illustrated in Figure~\ref{fig:K3noK2}. 
    Then $H$  
     will have a Braess $K_3$ containing no Braess edges if and only if  
    \begin{equation} \label{eq:K3noK2}
        3m^2+93m+360<6m\alpha<4m^2+94m+330.
    \end{equation}
\end{theorem}

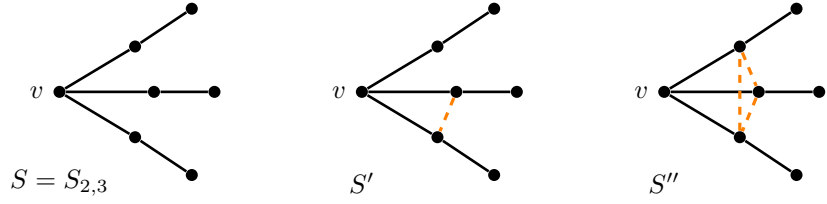
\begin{figure}
\begin{center}
\begin{tikzpicture}[vtx/.style={circle,fill,inner sep=1.6pt},
    bubble/.style={draw,very thick,circle,minimum size=1.5cm},
    ed/.style={line width=1pt},]

\begin{scope} 
  %\node[bubble] (B2) at (0,0) {};
  %\node at (B2.center) {$G$};
  \node[vtx,label=left:$v$] (a2) at (0.75,0) {};
  \node[vtx] (b2) at (2,0) {};
  \node[vtx] (c2) at (2.8,0) {};
  \node[vtx] (g2) at (1.75,0.6) {};
  \node[vtx] (h2) at (1.75,-0.6) {};
  \node[vtx] (d2) at (2.5,1.1) {};
  \node[vtx] (e2) at (2.5,-1.1) {};

  \draw[ed] (a2)--(b2);
  \draw[ed] (b2)--(c2);
  \draw[ed] (a2)--(g2);
  \draw[ed] (a2)--(h2);
  \draw[ed] (d2)--(g2);
  \draw[ed] (e2)--(h2);

  \node at (0.75,-1.2) {$S=S_{2,3}$};
\end{scope}

\begin{scope}[xshift=4cm]
  %\node[bubble] (B) at (0,0) {};
  %\node at (B.center) {$G$};
  \node[vtx,label=left:$v$] (a) at (0.75,0) {};
  \node[vtx] (b) at (2,0) {};
  \node[vtx] (c) at (2.8,0) {};
  \node[vtx] (g) at (1.75,0.6) {};
  \node[vtx] (h) at (1.75,-0.6) {};
  \node[vtx] (d) at (2.5,1.1) {};
  \node[vtx] (e) at (2.5,-1.1) {};

  \draw[ed] (a)--(b);
  \draw[ed] (b)--(c);
  \draw[ed] (a)--(g);
  \draw[ed] (a)--(h);
  \draw[ed] (d)--(g);
  \draw[ed] (e)--(h);

  \draw[very thick,orange, dashed] (b)--(h);

  \node at (0.75,-1.2) {$S'$};
\end{scope}

\begin{scope}[xshift=8cm]
  %\node[bubble] (B1) at (0,0) {};
  %\node at (B1.center) {$G$};
  \node[vtx,label=left:$v$] (a1) at (0.75,0) {};
  \node[vtx] (b1) at (2,0) {};
  \node[vtx] (c1) at (2.8,0) {};
  \node[vtx] (g1) at (1.75,0.6) {};
  \node[vtx] (h1) at (1.75,-0.6) {};
  \node[vtx] (d1) at (2.5,1.1) {};
  \node[vtx] (e1) at (2.5,-1.1) {};

  \draw[ed] (a1)--(b1);
  \draw[ed] (b1)--(c1);
  \draw[ed] (a1)--(g1);
  \draw[ed] (a1)--(h1);
  \draw[ed] (d1)--(g1);
  \draw[ed] (e1)--(h1);

  \draw[very thick,orange, dashed] (b1)--(h1)--(g1)--(b1);

  \node at (0.75,-1.2) {$S''$};
\end{scope}

\end{tikzpicture}
\end{center}
\caption{Graphs used in the proof of Theorem~\ref{thm:spider}.}\label{fig:spiderw}
\end{figure}

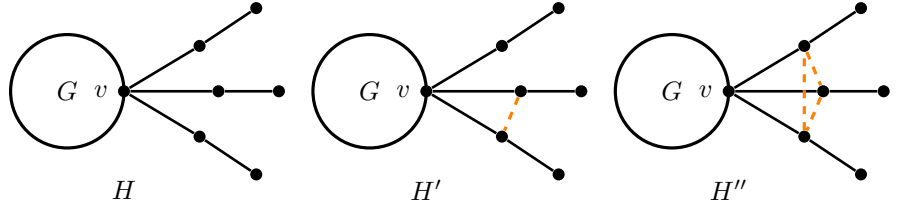
\begin{figure}
\begin{center}
\begin{tikzpicture}[vtx/.style={circle,fill,inner sep=1.6pt},
    bubble/.style={draw,very thick,circle,minimum size=1.5cm},
    ed/.style={line width=1pt},]

\begin{scope} 
  \node[bubble] (B2) at (0,0) {};
  \node at (B2.center) {$G$};
  \node[vtx,label=left:$v$] (a2) at (0.75,0) {};
  \node[vtx] (b2) at (2,0) {};
  \node[vtx] (c2) at (2.8,0) {};
  \node[vtx] (g2) at (1.75,0.6) {};
  \node[vtx] (h2) at (1.75,-0.6) {};
  \node[vtx] (d2) at (2.5,1.1) {};
  \node[vtx] (e2) at (2.5,-1.1) {};

  \draw[ed] (a2)--(b2);
  \draw[ed] (b2)--(c2);
  \draw[ed] (a2)--(g2);
  \draw[ed] (a2)--(h2);
  \draw[ed] (d2)--(g2);
  \draw[ed] (e2)--(h2);

  \node at (0.75,-1.3) {$H$};
\end{scope}

\begin{scope}[xshift=4cm]
  \node[bubble] (B) at (0,0) {};
  \node at (B.center) {$G$};
  \node[vtx,label=left:$v$] (a) at (0.75,0) {};
  \node[vtx] (b) at (2,0) {};
  \node[vtx] (c) at (2.8,0) {};
  \node[vtx] (g) at (1.75,0.6) {};
  \node[vtx] (h) at (1.75,-0.6) {};
  \node[vtx] (d) at (2.5,1.1) {};
  \node[vtx] (e) at (2.5,-1.1) {};

  \draw[ed] (a)--(b);
  \draw[ed] (b)--(c);
  \draw[ed] (a)--(g);
  \draw[ed] (a)--(h);
  \draw[ed] (d)--(g);
  \draw[ed] (e)--(h);

  \draw[very thick,orange, dashed] (b)--(h);

  \node at (0.75,-1.3) {$H'$};
\end{scope}

\begin{scope}[xshift=8cm]
  \node[bubble] (B1) at (0,0) {};
  \node at (B1.center) {$G$};
  \node[vtx,label=left:$v$] (a1) at (0.75,0) {};
  \node[vtx] (b1) at (2,0) {};
  \node[vtx] (c1) at (2.8,0) {};
  \node[vtx] (g1) at (1.75,0.6) {};
  \node[vtx] (h1) at (1.75,-0.6) {};
  \node[vtx] (d1) at (2.5,1.1) {};
  \node[vtx] (e1) at (2.5,-1.1) {};

  \draw[ed] (a1)--(b1);
  \draw[ed] (b1)--(c1);
  \draw[ed] (a1)--(g1);
  \draw[ed] (a1)--(h1);
  \draw[ed] (d1)--(g1);
  \draw[ed] (e1)--(h1);

  \draw[very thick,orange, dashed] (b1)--(h1)--(g1)--(b1);

  \node at (0.75,-1.3) {$H''$};
\end{scope}

\end{tikzpicture}
\end{center}
\caption{A graph $H$ having a Braess $K_3$, with none of the edges in the $K_3$ a Braess edge, when Equation~(\ref{eq:K3noK2}) is satisfied.}\label{fig:K3noK2}
\end{figure}

\begin{proof}
Let $G$ be any graph with $v\in V(G)$. 
Let $S, S'$ and $S''$ be the graphs in Figure~\ref{fig:spiderw}. Let
$H=G\oplus_v S, H'=G\oplus_v S'$ and
$H''=G\oplus_v S''$, as shown in Figure~\ref{fig:K3noK2}.
We will show that 
$\K(H')<\K(H)<\K(H'')$. 

Using the 1-separation formula in  Theorem~\ref{thm:1sep},   
$$\K(H) = \frac{m_S(\K(S)+\mu_G(v))+m_G(\K(G)+\mu_S(v))}{m_S+m_G}$$
with $m_S = 6,\ \K(S) = 9.5,$ and $ \mu_S(v) 
= 12,$ gives 
\begin{eqnarray*}
\K(H) 
&=& \frac{57 +6\mu_G(v)+m\K(G)+12m}{m+6}.
\end{eqnarray*}
Further 
$m_{S'} = 7,\ \K(S') = \frac{172}{21}$ and $\mu_{S'}(v) = 
\frac{34}{3},$
 hence 
$$\K(H') = \frac{\frac{172}{3}+7\mu_G(v)+m\K(G)+\frac{34}{3}m}{m+7}.$$
Finally, 
$m_{S''} = 9,\ \K(S'') = \frac{37}{6},$ and $\mu_{S''}(v)  
= \frac{21}{2},$
thus
$$\K(H'') = \frac{\frac{111}{2}+9\mu_G(v)+m\K(G)+\frac{21}{2}m}{m+9}.$$
Thus, using Equation~(\ref{eq:amu}), 
$\K(H)<\K(H'')$ 
if and only if
\begin{equation} \label{eq:upper}
    2m\alpha-m^2-31m-120>0.
\end{equation}
Further, 
$\K(H')<\K(H)$ 
if and only if 
\begin{equation}\label{eq:lower}
3m\alpha-2m^2-47m-165<0.
\end{equation}
Combining Equations~(\ref{eq:upper}) and (\ref{eq:lower}),
$H$ has a Braess $K_3$ as indicated with no edge of the coclique being a Braess edge if and only if
$$3m^2+93m+360<6m\alpha<4m^2+94m+330.$$
\end{proof}

Attaching a spider graph to a complete bipartite graph gives an infinite family of graphs that satisfy the inequalities in line (\ref{eq:K3noK2}):

\begin{corollary}\label{cor:infinite}
Suppose $a\geq 36$. Let $H$ be the graph obtained from 
$K_{a,3}$ by attaching the center of a spider graph $S_{2,3}$ to a vertex of degree 3 in $K_{a,3}$. Then $H$ will have a Braess $K_3$ containing no Braess edges.  
\end{corollary}

\begin{proof}
    Let $G=K_{a,b}$ with bipartition $V(G)=A\cup B$ such that $|A|=a$ and $|B|=b$. As per (\ref{eq:Kbip}), $\K(G)=a+b-\frac{3}{2}$ %(see \cite{KD}) 
    and by (\ref{eq:alphabip}),
$\alpha_G(v)=\frac{4a-3}{2}$ if $v\in A$ and
$\alpha_G(v)=\frac{4b-3}{2}$ if $v\in B$. %(see \cite[Theorem 4.1]{G2016}.)
The corollary now follows from Theorem~\ref{thm:spider}.
\end{proof}

Our next result, Theorem~\ref{thm:BraessEdgesAndNot} will demonstrate the following observation:

\begin{remark}
    There exist examples of Braess cliques of a graph $G$ for which some of the edges of the clique are Braess edges for $G$, but not all of the edges of the clique. See Figure~\ref{fig:G4etal} for an example.  
\end{remark}

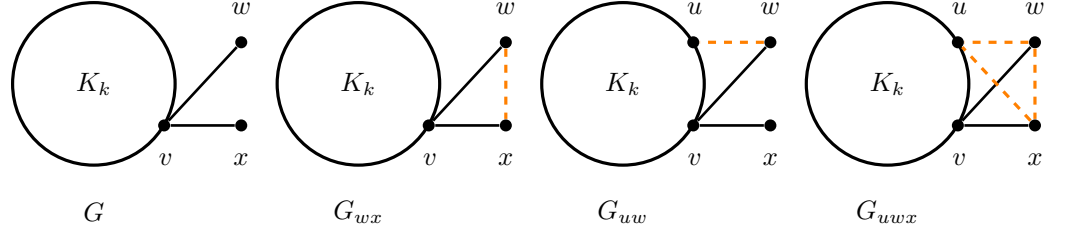
\begin{figure}
\begin{center}
\begin{tikzpicture}[vtx/.style={circle,fill,inner sep=1.6pt},
    sbubble/.style={draw,very thick,rectangle,minimum size=1.5cm},
    bubble/.style={draw,very thick,circle,minimum size=2.15cm},
    ed/.style={line width=1pt},]

\begin{scope}
\node[bubble] (B05) at (0,0) {};
  %\node[sbubble] (B0) at (0,0) {};
  \node at (B0.center) {$K_k$};
  \node[vtx] (v) at (0.93,-0.55) {}; % on the right bottom side of the bubble
  \node at (0.95, -1) {$v$};
  %\node[vtx] (w) at (0.93,0.55) {};
  %\node at (0.95, 1) {$w$};
  \node[vtx] (a) at (1.95, 0.55) {};
  \node[vtx] (b) at (1.95, -0.55) {};
  \node at (1.95, 1) {$w$}; 
  \node at (1.95,-1) {$x$};
     \draw[ed] (v)--(a);
    \draw[ed] (v)--(b);
  %  \draw[very thick,orange] (a)--(w);
  %   \draw[very thick,orange] (a)--(b);
  %  \draw[very thick,orange] (b)--(w);
\node at (0, -1.7) {$G$};
\end{scope}

 \begin{scope}[xshift=3.5cm]
  \node[bubble] (B05) at (0,0) {};
  %\node[sbubble] (B0) at (0,0) {};
  \node at (B05.center) {$K_k$};
  \node[vtx] (v) at (0.93,-0.55) {}; % on the right bottom side of the bubble
  \node at (0.95, -1) {$v$};
  %\node[vtx] (w) at (0.93,0.55) {};
 % \node at (0.95, 1) {$w$};
  \node[vtx] (a) at (1.95, 0.55) {};
  \node[vtx] (b) at (1.95, -0.55) {};
  \node at (1.95, 1) {$w$}; 
  \node at (1.95,-1) {$x$};
     \draw[ed] (v)--(a);
    \draw[ed] (v)--(b);
    %\draw[very thick,orange] (a)--(w);
     \draw[very thick,orange, dashed] (a)--(b);
   % \draw[very thick,orange] (b)--(w);
\node at (0, -1.7) {$G_{wx}$}; %{$G'$};
\end{scope}

\begin{scope}[xshift=7cm]
  \node[bubble] (B05) at (0,0) {};
  %\node[sbubble] (B0) at (0,0) {};
  \node at (B05.center) {$K_k$};
  \node[vtx] (v) at (0.93,-0.55) {}; % on the right bottom side of the bubble
  \node at (0.95, -1) {$v$};
  \node[vtx] (u) at (0.93,0.55) {};
  \node at (0.95, 1) {$u$};
  \node[vtx] (a) at (1.95, 0.55) {};
  \node[vtx] (b) at (1.95, -0.55) {};
  \node at (1.95, 1) {$w$}; 
  \node at (1.95,-1) {$x$};
   \draw[ed] (v)--(a);
   \draw[ed] (v)--(b);
    \draw[very thick,orange, dashed] (a)--(u);
    % \draw[very thick,orange] (a)--(b);
   % \draw[very thick,orange] (b)--(w);
\node at (0, -1.7) {$G_{uw}$}; % {$G_{3x}$};
\end{scope}

\begin{scope}[xshift=10.5cm]
  \node[bubble] (B05) at (0,0) {};
  %\node[sbubble] (B0) at (0,0) {};
  \node at (B05.center) {$K_k$};
  \node[vtx] (v) at (0.93,-0.55) {}; % on the right bottom side of the bubble
  \node at (0.95, -1) {$v$};
  \node[vtx] (u) at (0.93,0.55) {};
  \node at (0.95, 1) {$u$};
  \node[vtx] (a) at (1.95, 0.55) {};
  \node[vtx] (b) at (1.95, -0.55) {};
  \node at (1.95, 1) {$w$}; 
  \node at (1.95,-1) {$x$};
     \draw[ed] (v)--(a);
    \draw[ed] (v)--(b);
    \draw[very thick,orange, dashed] (a)--(u);
     \draw[very thick,orange, dashed] (a)--(b);
    \draw[very thick,orange, dashed] (b)--(u);
\node at (0, -1.7) {$G_{uwx}$}; %{$G_4$};
\end{scope}
\end{tikzpicture}
\end{center}
\caption{A graph $G$ with a Braess edge $\{w,x\}$, a non-Braess edge 
$\{w,u\}$ or $\{x,u\}$, %(as in $G_{uw}$) 
and a Braess $K_3$. %(as in $G_{uwx}$).
}\label{fig:G4etal}
\end{figure}

For any vertex $v$ of a complete graphs $K_n$ on $n$ vertices as well as a center vertex $v$ of a star graph $S_n=K_{1,n-1}$, the following values can be found in~\cite{faught20221}:
\begin{eqnarray} \label{rem:Kmoments}
    \K(K_n) = \frac{(n-1)^2}{n},&\rm{\ and\ }&  \mu_{K_n}(v) = \frac{2(n-1)^2}{n},\\ 
    \label{rem:Smoments}
    \K(S_n) = \frac{2n-3}{2},&\rm{\ and\ }&     \mu_{S_n}(v) = n-1. 
\end{eqnarray}
\begin{lemma} \label{lem:kappaG}
Let $G$ be the graph with $k+2$ vertices obtained from $K_k$ by inserting two pendent vertices adjacent to
some vertex $v$ of $K_k$, (that is, $G=K_k \oplus_v S_3$   where $v$ is the center vertex of $S_3$ and any vertex of $K_k$). Then 
$$\K{(G)}=\frac{k^4-k^3+9k^2-11k+8}{k^3-k^2+4k}.$$
\end{lemma}

\begin{proof}
The formula follows directly from 1-separation formula of Theorem~\ref{thm:1sep} as well as the known
Kemeny values and moments for the complete graph $K_k$ and star $S_3$ noted in (\ref{rem:Kmoments}) and (\ref{rem:Smoments}).
\end{proof}

\begin{lemma}\label{lem:kappaG4}
    Let $G_{uwx}$ be the graph on $k+2$ vertices obtained from $K_k$ by inserting two vertices that are both adjacent to the same two vertices in $K_k$, and these two vertices are adjacent to each other. Then
    $$\K(G_{uwx}) = \frac{24k^7 + 30k^6 + 318k^5 + 234k^4 - 890k^3 + 1151k^2 + 398k - 840}{4(6k^6 + 6k^5 + 42k^4 + 114k^3 - 23k^2 - 70k)}.$$
\end{lemma}

\begin{proof}
We observe that $G_{uwx}$ has the transition matrix
$$T=
\renewcommand{\arraystretch}{2} % Default value: 1
\left[ \begin{array}{c|c|c} 
\frac{1}{k-1}(J-I)_{k-2} &\frac{1}{k-1}J_{k-2,2}& O_{k-2,2} \\ \hline
\frac{1}{k+1}J_{2,k-2}& \frac{1}{k+1}(J-I)_2& \frac{1}{k+1}J_2\\ \hline
O_{2,k-2}&\frac{1}{3}J_2&\frac{1}{3}(J-I)_2 
\end{array}\right]$$
\renewcommand{\arraystretch}{1} % Default value: 1
Then $T$ has $k-3$ eigenvectors  $e_i-e_{k-2}$, $q\leq i \leq k-3$ each with eigenvalue $\frac{-1}{k-1}$, as well as one eigenvector
$e_{k+1}-e_{k+2}$ with eigenvalue $\frac{-1}{3}$, and one
eigenvector $e_{k-1}-e_{k}$ with eigenvalue $\frac{-1}{k+1}$. In addition, $T$ has three eigenvalues from the quotient matrix 
$$Q=\renewcommand{\arraystretch}{1.5} % Default value: 1
\left[ \begin{array}{ccc}
\frac{k-3}{k-1}&\frac{2}{k-1}&0\\
\frac{k-2}{k+1}& \frac{1}{k+1}& \frac{2}{k+1}\\
0&\frac{2}{3}&\frac{1}{3}
\end{array}\right].\renewcommand{\arraystretch}{1} % Default value: 1
$$
Now $Q$ has eigenvalue 1 as well as the two eigenvalues
$$\frac{k^2-3k-10\pm \sqrt{k^4+54k^3-167k^2+156}}{6(k^2-1)}. $$
We obtain $\K(G_{uwx})$ by applying the eigenvalue formula (\ref{eq:Keig}). 
\end{proof}

Next we intend to develop Kemeny's constant for $G_{uw}$. As an intermediate step, we determine Kemeny's constant and a moment of $G_3$ in Figure~\ref{fig:G3}.

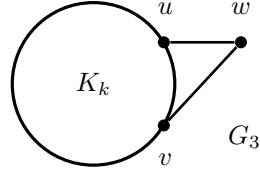
\begin{figure}[h]
\begin{center}
  \begin{tikzpicture}[vtx/.style={circle,fill,inner sep=1.6pt},
    sbubble/.style={draw,very thick,rectangle,minimum size=1.5cm},
    bubble/.style={draw,very thick,circle,minimum size=2.15cm},
    ed/.style={line width=1pt},]

\begin{scope}%[xshift=7cm]
  \node[bubble] (B05) at (0,0) {};
  %\node[sbubble] (B0) at (0,0) {};
  \node at (B05.center) {$K_k$};
  \node[vtx] (v) at (0.93,-0.55) {}; % on the right bottom side of the bubble
  \node at (0.95, -1) {$v$};
  \node at (1.95, 1) {$w$};
  \node[vtx] (u) at (0.93,0.55) {};
  \node at (0.95, 1) {$u$};
  \node[vtx] (a) at (1.95, 0.55) {};
  %\node[vtx] (b) at (1.95, -0.55) {};
   \draw[ed] (v)--(a);
   % \draw[ed] (v)--(b);
    \draw[ed] (a)--(u);
    % \draw[very thick,orange] (a)--(b);
   % \draw[very thick,orange] (b)--(w);
\node at (2, -0.7) {$G_3$};
\end{scope}
\end{tikzpicture}
\end{center}
\caption{An intermediate graph $G_3$.}\label{fig:G3}
\end{figure}

\begin{lemma}\label{lem:kappaG3}
If $G_3$ is the graph on $k+1$ vertices obtained the graph $K_k$ by inserting a vertex adjacent to exactly two vertices of $K_k,$ then
$$\kappa(G_3)=\frac{k^5-k^4+6k^3-4k^2-6k+12}{k^4+3k^2+4k}.$$
\end{lemma}

\begin{proof}
The transition matrix for $G_3$ is  
$$T=\renewcommand{\arraystretch}{2} % Default value: 1
\left[\begin{array}{c|c|c}
\frac{1}{k-1}(J-I)_{k-2}&\frac{1}{k-1}J_{k-2,2}&O_{k-2,1}\\ \hline
\frac{1}{k}J_{2,k-2}& \frac{1}{k}(J-I)_2 & \frac{1}{k}J_{2,1} \\ \hline
O_{1,k-2}&\frac{1}{2}J_{1,2}&0
\end{array}\right].$$
The matrix $T$  has one eigenvector
$e_{k-1}-e_{k}$ with eigenvalue $-\frac{1}{k}$ and $k-3$ eigenvectors 
$e_i-e_{k-2}$ for $1\leq i\leq k-3$ with eigenvalue $-\frac{1}{k-1}$ as well
as the eigenvalues of the quotient matrix
$$
Q=\renewcommand{\arraystretch}{1.5} % Default value: 1
\left[ \begin{array}{ccc}
\frac{k-3}{k-1}&\frac{2}{k-1}&0 \\
\frac{k-2}{k} & \frac{1}{k} & \frac{1}{k} \\
0&1&0
\end{array}\right].
\renewcommand{\arraystretch}{1} % Default value: 1
$$
The eigenvalues of $Q$ are 1 as well as the two eigenvalues
$$ -\frac{k+1 \pm \sqrt{4k^3 - 15k^2 + 14k + 1}}{2(k^2 - k)}. $$
Using equation (\ref{eq:Keig}) the formula for $\K(G_3)$ is obtained.
\end{proof}

To calculate moments of $G_3$,  the following theorem of Cayley (see e.g, \cite{takac}) is useful.

\begin{theorem}\label{thm:Cayley}{\rm{[Cayley]}} The number of spanning forests of $K_n$ with $t$ distinct trees such that $t$ specific vertices are in 
different trees is $tn^{n-t-1}.$
\end{theorem}

\begin{lemma}\label{lem:tau}
If $G_3$ is the graph on $k+1$ vertices obtained from the graph $K_k$ by inserting a vertex $w$ adjacent to exactly two vertices, say $u$ and $v$,  of $K_k,$ then
$\tau(G_3)=2k^{k-3}(k+1).$    
\end{lemma}

\begin{proof}
    Note that a spanning forest could include both edges $\{w,v\}$ and $\{w,u\}$, or just one of these edges.
In the first case, any spanning tree that contains both of these edges can be constructed from a unique spanning forest of $K_k$ containing exactly two trees, with $u$ in one tree and $v$ in the other. Thus by Theorem~\ref{thm:Cayley}, there are $2k^{k-3}$ such trees. 

In the second case, suppose a spanning tree of $G_3$ contains the edge $\{w,v\}$ but not $\{ w,u\}$. The number of such trees is simply the number of spanning trees of $K_k$, namely $k^{k-2}$, by Theorem~\ref{thm:Cayley}. By symmetry, the same count will be obtained if $G_3$ contains $\{w,u\}$ but not $\{w,v\}.$
Therefore $\tau(G_3)=2k^{k-3}+2k^{k-2}=2k^{k-3}(k+1).$
\end{proof}

\begin{lemma}\label{lem:momentG3}
    If $G_3$ is the graph on $k+1$ vertices obtained the graph $K_k$ by inserting a vertex $w$ adjacent to exactly two vertices, say $u$ and $v$,  of $K_k.$ Then
    $$\mu_{G_3}(v)=
    \frac{4k^3-3k^2+3k+6}{2k(k+1)}.
    $$
\end{lemma}

\begin{proof}
Focusing on formula (\ref{eq:momentTree}), we  determine the entries $f_{i,v}$ of $F_v$,  corresponding to vertex $v_i$.

Suppose $v_i=u$. Then, in any spanning $2$-tree forest of $G_3$ separating $u$ and $v$, either vertex $w$ is in the tree with vertex $u$ or in the tree with vertex $v$. Hence $f_{i,v}$ is twice the number of spanning $2$-trees of $K_k$ separating $u$ and $v$.  By Theorem~\ref{thm:Cayley}, $f_{i,v}=4k^{k-3}.$

Suppose $v_i=w$. Any spanning $2$-tree forest of $G_2$ which separates $w$ and $v$ could either include edge $\{w,u\}$ or not. In the former case, the number of spanning 2-tree forests of $G_3$ separating $v$ and $w$ is the same as the number of spanning $2$-tree forest of $K_k$ separating $v$ and $u$ (which 
is $2k^{k-3}$ by Theorem~\ref{thm:Cayley}). In the latter case, the number of spanning $2$-tree forests is simply the number of spanning trees of $K_k.$ Therefore $f_{i,v}=k^{k-3}(2+k).$

Suppose $v_i\neq w$ and $v_i\neq u$. A spanning 2-tree forest of $G_3$ separating $v_i$ and $v$ could contain both edges $\{v,w\}$ and $\{w,u\}$, or just one of of these edges. The number of spanning $2$-tree forests containing both edges separating $v_i$ and $v$ will be equal to the number of spanning $3$-tree forests of $G_3$ separating $v_i$, $v$ and $u$ (that is, $3k^{k-4}$, by Theorem~\ref{thm:Cayley}). The number of spanning 2-tree forests of $G_3$ separating $v_i$ and $v$ that includes exactly one of the edges
$\{v,w\}$ or $\{w,u\}$ would be the number of spanning 2-tree forests of $K_k$ separating $v_i$ and $v$ (that is, $2k^{k-3}$ for each edge, by Theorem~\ref{thm:Cayley}). Thus $f_{i,v}=3k^{k-4}+2(2k^{k-3})=k^{k-4}(4k+3).$

We note that the degree of $v_i$ is $k$ if $v_i=u$, the degree is 2 if $v_i=w$ and the degree is $k-1$ for the $k-2$ vertices  $v_i\not\in \{u,w\}$.
Therefore, by (\ref{eq:moment}) and Lemma~\ref{lem:tau}, 
\begin{eqnarray*}
    \mu_G(v)&=&
\frac{1}{2k^{k-3}(k+1)}\left[ 
k(4k^{k-3})+2k^{k-3}(2+k)+(k-2)(k-1)k^{k-4}(4k+3)\right] 
\\[1em]
&=& \frac{4k^3-3k^2+3k+6}{2k(k+1)}.
\end{eqnarray*}  
\end{proof}

\begin{lemma}\label{lem:G3x}
    If $G_{uw}$ is the graph obtained from $G_3$ by inserting a pendent vertex $x$ adjacent to $v$,
    then 
    $$\K(G_{uw})=
\frac{k^5+10k^3-3k^2+2k+18}{k^4+5k^2+6k}.
    $$
\end{lemma}

\begin{proof}
    Noting that $G_{uw}= G_3 \oplus_v \{v,x\}$,
    using the 1-separation 
    Theorem~\ref{thm:1sep},
    $$\K(G_{uw})=\frac{\left[ \binom{k}{2}+2\right]\left[ \K(G_3)+1\right] +1\left[ \frac{1}{2} + \mu_{G_3}(v) \right]}{\binom{k}{2} +2+1}
    $$
since $\K(\{v,x\})=\K(K_2)=\frac{1}{2}$ by (\ref{eq:moment}). 
The result follows from 
Lemma~\ref{lem:kappaG3} and Lemma~\ref{lem:momentG3}.
\end{proof}

\begin{theorem} \label{thm:BraessEdgesAndNot}
If $k\geq 6$ and $G$ is the graph obtained from $K_k$ by inserting two pendent vertices %of degree 1 
to some vertex of $K_k$, then $G$ has a Braess $K_3$, with the property that one edge of the $K_3$ is a Braess edge in $G$ but not the other two edges. 
\end{theorem}

\begin{proof}
To establish that $G$ has a Braess $K_3$, it is sufficient to show that the $\K(G_{uwx})>\K(G)$.
By Lemmas~\ref{lem:kappaG4} and \ref{lem:kappaG}, 
$\K(G_{uwx})-\K(G)=$
\[
\frac{(6k^8 + 24k^7 - 204k^6 - 744k^5) + (717k^4 - 93k^3) + (1226k^2 + 88k - 1120)}{4(6k^8 + 60k^6 + 96k^5 + 31k^4 + 409k^3 - 22k^2 - 280k}.
\]
The numerator is partitioned into three summands, each of which is positive for $k\geq 6$ and the denominator is also positive for $k\geq 6$. Thus $G$ has a Braess $K_3$. 

The edge incident to the pendent twins of $G$ is a Braess edge in $G$ by Corollary~\ref{cor:pendent}. 
To show that the two remaining edges of the Braess $K_3$ are not Braess edges in $G$ it is sufficient to show that
$\K(G_{uw})\leq \K(G).$
By Lemma~\ref{lem:kappaG} and Lemma~\ref{lem:G3x},
\[
\K(G)-\K(G_{uw})
=
\frac{3k^4 +2k^3 - 5k^2 - 16k - 24}{k^6 - k^5 + 9k^4 + k^3 + 14k^2 + 24k}.
\]
The denominator is positive for $k>1$ and the numerator is positive for $k>2$. 
\end{proof}

\section{Concluding notes}

In this paper we have explored the phenomena of graphs having a Braess clique. We provided several classes of graphs and constructions that realize a Braess clique. 
We noted that, 
given any fixed $\ell\geq 2$, almost every connected planar labelled graph has a Braess $K_\ell.$  We also observed that there is a nontrivial relationship between Braess cliques of order at least three, and Braess edges.

There are various ways that the accessibility index of a vertex shows up in various formulas and bounds when considering Braess cliques. 
We expect that it is worth exploring further information about the accessibility index and its role in Kemeny's constant in the future. Other work can also be done to explore the amount Kemeny's constant changes when adding an edge or larger clique to families of graphs with large independent sets.

\bigskip 
\noindent 
\textbf{Acknowledgement:}
KVM thanks Nicholas Paiement for initial conversations about   Kemeny's constant with respect to cliques in graphs. 

\appendix
\section{Appendix on stochastic complements}\label{app:stoch_comp}
In this appendix, we briefly outline the notation and results involving stochastic complements for reference when used in Proposition~\ref{prop:acc_quotient}. Stochastic complements are a means by which the stationary vector and mean first passage matrix for an irreducible Markov chain may be found using a divide-and-conquer approach. For further information and results, see \cite{stevesbook} or \cite{kirkland2001divide}.

Let $T$ be the $n\times n$ transition matrix of an irreducible Markov chain, block partitioned as follows:
\[T = \left[\begin{array}{c|c} T_{11} & T_{12} \\ \hline T_{21} & T_{22}\end{array}\right],\]
where $T_{11}$ is an $\ell\times \ell$ matrix and $T_{22}$ is of order $n-\ell$.
The \emph{stochastic complements} corresponding to this partition are the stochastic matrices 
\[S_1 := T_{11} + T_{12}(I-T_{22})^{-1}T_{21};\]
\[S_2 := T_{22} + T_{21}(I-T_{11})^{-1}T_{12},\]
where the orders of the identity matrices above (and in the expressions below) can be deduced from context. Note that $S_1$ and $S_2$ are themselves irreducible stochastic matrices. 

Let the stationary vector $w$ and mean first passage matrix $M$ of $T$ be partitioned conformally,   \[w^\top = \left[\begin{array}{c|c} w_1^\top & w_2^\top \end{array}\right]\]
and \[M = \left[\begin{array}{c|c} M_{11} & M_{12} \\ \hline M_{21} & M_{22}\end{array}\right].\]
 If $u_i$ is the stationary vector of $S_i$, then $w_i = a_iu_i$, where $\begin{bmatrix} a_1 & a_2 \end{bmatrix}$ is the stationary vector of the $2\times 2$ coupling matrix 
\[\left[\begin{array}{cc} u_1^\top T_{11}\mathbbm{1} & u_1^\top T_{12}\mathbbm{1} \\ u_2^\top T_{21} \mathbbm{1} & u_2^\top T_{22} \mathbbm{1}\end{array}\right].\]
Furthermore,  given $\gamma_i = \frac{1}{\mathbbm{1}^\top w_i}$,
\[M_{11} = \gamma_1 M_{S_1} + V_1,\]
\[M_{22} = \gamma_2 M_{S_2} + V_2,\]
where $M_{S_i}$ refers to the mean first passage matrix for the Markov chain represented by $S_i$, and 
\[V_1 = (I-S_1)^\#T_{12}(I-T_{22})^{-1}J_{n-\ell, \ell} -[(I-S_1)^\#T_{12}(I-T_{22})^{-1}J_{n-\ell, \ell}]^\top,\]
\[V_2 = (I-S_2)^\#T_{21}(I-T_{11})^{-1}J_{\ell, n-\ell} -[(I-S_2)^\#T_{21}(I-T_{11})^{-1}J_{\ell, n-\ell}]^\top.\]
Furthermore, 
\[M_{12} = (I-T_{11})^{-1}T_{12}(M_{22} - (M_{22})_{dg}) + (I-T_{11})^{-1}J_{\ell, n-\ell},\]
\[M_{21} = (I-T_{22})^{-1}T_{21}(M_{11} - (M_{11})_{dg}) + (I-T_{22})^{-1}J_{n-\ell, \ell}.\]

%%%%%%%%%%%%%%%%%%%%%%%%%%%%%%%%%%%%%%%%%%%%%%%%%%%%%%%%%%%%%


\begin{thebibliography}{1}
\bibitem{Breen2022Bridges}
Jane Breen, Emanuele Crisostomi, and Sooyeong Kim.
\newblock Kemeny's constant for a graph with bridges.
\newblock {\em Discrete Appl. Math.}, 322:20--35, 2022.

\bibitem{breen2025threshold}
Jane Breen, Sooyeong Kim, Alexander~Low Fung, Amy Mann, Andrei~A Parfeni, and Giovanni Tedesco.
\newblock Threshold graphs, {K}emeny's constant, and related random walk parameters.
\newblock {\em Linear Algebra Appl.}, 709:284--313, 2025.

\bibitem{brouwer2012spectra}
Andries~E. Brouwer and Willem~H. Haemers.
\newblock {\em {Spectra of Graphs}}.
\newblock Universitext. Springer, New York, 2012.

\bibitem{CL}
G.~Chartrand and L.~Lesniak.
\newblock {\em Graphs and Digraphs}.
\newblock Chapman and Hall, 2005.

\bibitem{chebotarev2020hitting}
Pavel Chebotarev and Elena Deza.
\newblock Hitting time quasi-metric and its forest representation.
\newblock {\em Optim. Lett.}, 14(2):291--307, 2020.

\bibitem{ciardo2020braess}
Lorenzo Ciardo.
\newblock The {B}raess' paradox for pendent twins.
\newblock {\em Linear Algebra Appl.}, 590:304--316, 2020.

\bibitem{ciardo2022kemeny}
Lorenzo Ciardo, Geir Dahl, and Steve Kirkland.
\newblock On {K}emeny's constant for trees with fixed order and diameter.
\newblock {\em Linear Multilinear Algebra}, 70(12):2331--2353, 2022.

\bibitem{faught2020resistance}
Nolan Faught, Mark Kempton, and Adam Knudson.
\newblock Resistance distance, {K}irchhoff index, and {K}emeny's constant in flower graphs.
\newblock {\em MATCH Commun. Math. Comput. Chem.}, 86:405--427, 2021.

\bibitem{faught20221}
Nolan Faught, Mark Kempton, and Adam Knudson.
\newblock A 1-separation formula for the graph {K}emeny constant and {B}raess edges.
\newblock {\em J. Math. Chem.}, 60(1):49--69, 2022.

\bibitem{G2016}
Severino~V. Gervacio.
\newblock Resistance distance in complete n-partite graphs.
\newblock {\em Discrete Appl. Math.}, 203:53--61, 2016.

\bibitem{hu2019complete}
Yuxiang Hu and Steve Kirkland.
\newblock Complete multipartite graphs and {B}raess edges.
\newblock {\em Linear Algebra Appl.}, 579:284--301, 2019.

\bibitem{jang2025kemeny}
Jihyeug Jang, Mark Kempton, Sooyeong Kim, Adam Knudson, Neal Madras, and Minho Song.
\newblock Kemeny's constant and enumerating {B}raess edges in trees.
\newblock {\em Linear Multilinear Algebra}, 73(8):1529--1565, 2025.

\bibitem{kemenysnell}
John~G. Kemeny and J.~Laurie Snell.
\newblock {\em Finite {M}arkov {C}hains}.
\newblock The University Series in Undergraduate Mathematics. D. Van Nostrand Co., Inc., Princeton, N.J.-Toronto-London-New York, 1960.

\bibitem{kim2022families}
Sooyeong Kim.
\newblock Families of graphs with twin pendent paths and the {B}raess edge.
\newblock {\em Electron. J. Linear Algebra}, 38:9--31, 2022.

\bibitem{stevesbook}
Stephen~J Kirkland and Michael Neumann.
\newblock {\em Group inverses of M-matrices and their applications}.
\newblock CRC Press Boca Raton, FL, 2013.

\bibitem{kirkland2001divide}
Stephen~J Kirkland, Michael Neumann, and Jianhong Xu.
\newblock A divide and conquer approach to computing the mean first passage matrix for {M}arkov chains via {P}erron complement reductions.
\newblock {\em Numer. Linear Algebra Appl.}, 8(5):287--295, 2001.

\bibitem{kirkland2016randomwalk}
Steve Kirkland.
\newblock Random walk centrality and a partition of {K}emeny's constant.
\newblock {\em Czech. Math. J.}, 66(141):757--775, 2016.

\bibitem{kirkland2016kemeny}
Steve Kirkland and Ze~Zeng.
\newblock Kemeny's constant and an analogue of {B}raess' paradox for trees.
\newblock {\em Electron. J. Linear Algebra}, 31:444--464, 2016.

\bibitem{KD}
R.E. Kooij and J.L. Dubbeldam.
\newblock Kemeny's constant for several families of graphs and real-world networks.
\newblock {\em Discrete Appl. Math.}, 25:96--107, 2020.

\bibitem{takac}
Tak\'acs Lajos.
\newblock On {C}ayley's formula for counting forests.
\newblock {\em J. Comb. Theory, Ser. A}, 53(2):321--323, 1990.

\bibitem{levene2002kemeny}
Mark Levene and George Loizou.
\newblock Kemeny's constant and the random surfer.
\newblock {\em Amer. Math. Monthly}, 109(8):741--745, 2002.

\bibitem{meyer1975role}
Carl~D. Meyer, Jr.
\newblock The role of the group generalized inverse in the theory of finite {M}arkov chains.
\newblock {\em SIAM Rev.}, 17:443--464, 1975.

\bibitem{noh2004random}
Jae~Dong Noh and Heiko Rieger.
\newblock Random walks on complex networks.
\newblock {\em Phys. Rev. Lett.}, 92(11):118701, 2004.

\bibitem{seneta}
Eugene Seneta.
\newblock {\em {Non-negative Matrices and Markov Chains}}.
\newblock Springer-Verlag, New York, 1981.

\bibitem{wang2017kemeny}
Xiangrong Wang, Johan L.~A. Dubbeldam, and Piet Van~Mieghem.
\newblock Kemeny's constant and the effective graph resistance.
\newblock {\em Linear Algebra Appl.}, 535:231--244, 2017.

\end{thebibliography}
\end{document}